\documentclass{article}

\usepackage[utf8]{inputenc}

\usepackage{mathtools}
\usepackage{amsthm}
\usepackage{todonotes}
\usepackage{xifthen}
\usepackage{amssymb}
\usepackage{hyperref}
\hypersetup{
    colorlinks=true,
    linkcolor=blue,
    citecolor=blue,
    urlcolor=blue
}
\usepackage{zref-clever}
\usepackage{tikz}
\usetikzlibrary{calc}
\tikzset{
	plainvertex/.style={circle, draw, inner sep=1.6pt},
	rootvertex/.style={circle, draw, fill=orange!70!red, inner sep=1.3pt},
	vvertex/.style={circle, draw, fill=blue!55!black, inner sep=1.6pt},
	wvertex/.style={circle, draw, fill=orange!70!red, inner sep=1.6pt},
	lbl/.style={font=\small},
}
\usepackage{subcaption}
\usepackage[backend=biber,style=alphabetic,maxnames=8,maxalphanames=8,sorting=anyt]{biblatex}
\usepackage{aliascnt}

\newtheorem{theorem}{Theorem}[section]

\newaliascnt{definition}{theorem}
\newtheorem{definition}[definition]{Definition}
\aliascntresetthe{definition}

\newaliascnt{corollary}{theorem}
\newtheorem{corollary}[corollary]{Corollary}
\aliascntresetthe{corollary}

\newaliascnt{proposition}{theorem}
\newtheorem{proposition}[proposition]{Proposition}
\aliascntresetthe{proposition}

\newaliascnt{lemma}{theorem}
\newtheorem{lemma}[lemma]{Lemma}
\aliascntresetthe{lemma}

\newtheorem{notation}{Notation}

\newaliascnt{claim}{theorem}

\aliascntresetthe{claim}

\newaliascnt{observation}{theorem}

\aliascntresetthe{observation}

\newaliascnt{fact}{theorem}

\aliascntresetthe{fact}

\newaliascnt{question}{theorem}

\aliascntresetthe{question}

\newaliascnt{remark}{theorem}
\newtheorem{remark}[remark]{Remark}
\aliascntresetthe{remark}

\DeclarePairedDelimiter{\pars}{(}{)}
\DeclarePairedDelimiter{\set}{\{}{\}}
\DeclarePairedDelimiter{\card}{|}{|}
\DeclarePairedDelimiter{\coeff}{[}{]}

\bibliography{references}

\newcommand{\defeq}{\coloneqq}

\newcommand{\N}{\mathbb{N}}%{N}
\newcommand{\R}{\mathbb{R}}

\newcommand{\disjointUnion}{\mathop{\sqcup}\displaylimits}
\newcommand{\largeDisjointUnion}{\mathop{\bigsqcup}\displaylimits}

\newcommand{\graphF}{\ensuremath{F}}
\newcommand{\graphG}{\ensuremath{G}}
\newcommand{\graphH}{\ensuremath{H}}
\newcommand{\graphI}{\ensuremath{I}}
\newcommand{\graphJ}{\ensuremath{J}}
\newcommand{\vertices}{\ensuremath{V}}
\newcommand{\edges}{\ensuremath{E}}
\newcommand{\numberOfEdges}{\ensuremath{k}}
\newcommand{\numberOfVertices}{\ensuremath{n}}

\newcommand{\successorClosedSubsets}[1]{\ensuremath{\mathcal{U}\pars*{#1}}}

\newcommand{\vertexT}{\ensuremath{t}}
\newcommand{\vertexU}{\ensuremath{u}}
\newcommand{\vertexV}{\ensuremath{v}}
\newcommand{\vertexW}{\ensuremath{w}}

\newcommand{\rev}[1]{\overleftarrow{#1}}
\newcommand{\starGraph}[1]{\ensuremath{S_{#1}}}
\newcommand{\cycleRootedStar}[1]{\ensuremath{{S_{#1}}}}
\newcommand{\connect}[2]{\ensuremath{#1 \to #2}}

\newcommand{\concentratedCycleRootedStar}[2]{\ensuremath{S_{#1, #2}}}
\newcommand{\subsetOfVertices}{Q}
\newcommand{\pathGraph}[1]{P_{#1}}
\newcommand{\cycleRootedPath}[1]{P_{#1}}
\newcommand{\concentratedCycleRootedPath}[2]{P_{#1, #2}}
\newcommand{\pathLength}{p}
\newcommand{\pathLengths}{\mathbf{p}}

\newcommand{\vertexrootedtree}{%
	\def\leaveslist{2,1,3}% number of leaves under each level-1 node
	\def\R{1.2}% radius of level 1
	\def\Rr{2.2}% radius of level 2
	\node[rootvertex] (root) at (0,0) {};
	\pgfmathsetmacro{\n}{3}% = length of \leaveslist
	\foreach [count=\i] \deg in \leaveslist {
		\pgfmathsetmacro{\ang}{90 + (\i-1)*360/\n - 360/(2*\n)}
		\node[plainvertex] (u\i) at (\ang:\R) {};
		\draw[->] (u\i) -- (root);
		\pgfmathtruncatemacro{\degint}{\deg}
		\ifnum\degint>0
			\foreach \j in {1,...,\degint} {
				\pgfmathsetmacro{\spread}{40}
				\pgfmathsetmacro{\suba}{\ang - \spread/2 + (\j-1)*\spread/max(\degint-1,1)}
				\node[plainvertex] (v\i-\j) at (\suba:\Rr) {};
				\draw[->] (v\i-\j) -- (u\i);
			}
		\fi
	}
}

\newcommand{\cyclerootedtree}{%
	\def\attachlist{2,0,1,3,1}% number of direct children of u_1..u_5
	\def\Rc{1.4}% cycle radius
	\def\Ra{2.6}% radius for direct children
	\pgfmathsetmacro{\ell}{5}
	\foreach [count=\i] \deg in \attachlist {
		\pgfmathsetmacro{\ang}{90 - (\i-1)*360/\ell}
		\node[rootvertex] (u\i) at (\ang:\Rc) {};
	}
	\foreach \i in {1,...,4} {
		\pgfmathtruncatemacro{\ni}{\i+1}
		\draw[->] (u\i) -- (u\ni);
	}
	\draw[->] (u5) -- (u1);
	\foreach [count=\i] \deg in \attachlist {
		\pgfmathsetmacro{\ang}{90 - (\i-1)*360/\ell}
		\pgfmathtruncatemacro{\degint}{\deg}
		\ifnum\degint>0
			\foreach \j in {1,...,\degint} {
				\pgfmathsetmacro{\spread}{50}
				\pgfmathsetmacro{\suba}{\ang - \spread/2 + (\j-1)*\spread/max(\degint-1,1)}
				\node[plainvertex] (a\i-\j) at (\suba:\Ra) {};
				\draw[->] (a\i-\j) -- (u\i);
			}
		\fi
	}
	\node[plainvertex] (b1) at ($(a1-1)+(90-18:1.0)$) {};
	\draw[->] (b1) -- (a1-1);
}

\newcommand{\siblingLiftCommon}{%
	\node[plainvertex] (r) at (0,2.4) {};
	\node[lbl] at (0.35,2.4) {$r$};
	\node[wvertex] (w) at (0,1.2) {};
	\node[lbl] at (0.4,1.2) {$\vertexW$};
	\node[vvertex] (v) at (1.4,0.2) {};
	\node[lbl] at (1.75,0.2) {$\vertexV$};
	\node[plainvertex] (z) at (-1.4,0.2) {};
	\node[plainvertex] (u1) at (-0.9,-1.2) {};
	\node[plainvertex] (u2) at (0,-1.5) {};
	\node[plainvertex] (u3) at (0.9,-1.2) {};
	\draw[->] (w) -- (r);
	\draw[->] (v) -- (w);
	\draw[->] (z) -- (w);
}

\newcommand{\moveLiftBefore}{%
	\node[wvertex] (w) at (90:1.4) {};
	\node[lbl] at ($(90:1.4)+(0.4,0.15)$) {$\vertexW$};
	\node[vvertex] (v) at (210:1.4) {};
	\node[lbl] at ($(210:1.4)+(-0.4,-0.05)$) {$\vertexV$};
	\node[plainvertex] (p) at (330:1.4) {};
	\draw[->] (v) -- (w);
	\draw[->] (w) -- (p);
	\draw[->] (p) -- (v);
	\node[plainvertex] (aw) at ($(90:1.4)+(90:1.1)$) {};
	\draw[->] (aw) -- (w);
	\node[plainvertex] (av1) at ($(210:1.4)+(160:1.1)$) {};
	\node[plainvertex] (av2) at ($(210:1.4)+(230:1.1)$) {};
	\draw[->] (av1) -- (v);
	\draw[->] (av2) -- (v);
	\node[plainvertex] (ap) at ($(330:1.4)+(-30:1.1)$) {};
	\draw[->] (ap) -- (p);
}

\newcommand{\moveLiftAfter}{%
	\node[wvertex] (w) at (90:1.4) {};
	\node[lbl] at ($(90:1.4)+(0.4,0.15)$) {$\vertexW$};
	\node[vvertex] (v) at (210:1.4) {};
	\node[lbl] at ($(210:1.4)+(-0.4,-0.05)$) {$\vertexV$};
	\node[plainvertex] (p) at (330:1.4) {};
	\draw[->] (v) -- (w);
	\draw[->] (w) -- (p);
	\draw[->] (p) -- (v);
	\node[plainvertex] (aw) at ($(90:1.4)+(70:1.1)$) {};
	\draw[->] (aw) -- (w);
	\node[plainvertex] (av1) at ($(90:1.4)+(110:1.1)$) {};
	\node[plainvertex] (av2) at ($(90:1.4)+(150:1.1)$) {};
	\draw[->] (av1) -- (w);
	\draw[->] (av2) -- (w);
	\node[plainvertex] (ap) at ($(330:1.4)+(-30:1.1)$) {};
	\draw[->] (ap) -- (p);
}

\newcommand{\vertexSymbol}{\circ}
\newcommand{\cycleSymol}{{\scalebox{0.6}{$\circlearrowleft$}}}

\newcommand{\cycleLength}[1]{\ensuremath{\ifthenelse{\isempty{#1}}{\ell}{\ell_{#1}}}}
\newcommand{\treeRootedAt}[1]{{\vertexRootedtree{#1}}}
\newcommand{\cycle}[1]{\ensuremath{\ifthenelse{\isempty{#1}}{C}{C_{#1}}}}
\newcommand{\tree}[1]{\ensuremath{\ifthenelse{\isempty{#1}}{T}{T_{#1}}}}
\newcommand{\cycleRootedtree}[1]{\ensuremath{\tree{#1}^{\cycleSymol}}}
\newcommand{\vertexRootedtree}[1]{\ensuremath{\tree{#1}^{\vertexSymbol}}}
\newcommand{\treeSize}[1]{\ensuremath{\ifthenelse{\isempty{#1}}{\partitionElement}{\partitionElement_{#1}}}}
\newcommand{\treeInternalSize}[1]{\ensuremath{\ifthenelse{\isempty{#1}}{s}{s_{#1}}}}
\newcommand{\treeInternalSizes}[1]{\ensuremath{\ifthenelse{\isempty{#1}}{\mathbf{s}}{\mathbf{s}_{#1}}}}
\newcommand{\vertexRootedTreeSize}[1]{\ensuremath{\ifthenelse{\isempty{#1}}{\partitionElement^{\vertexSymbol}}{\partitionElement^{\vertexSymbol}_{#1}}}}
\newcommand{\vertexRootedTreeSizes}{\mathbf{\partitionElement}^{\vertexSymbol}}
\newcommand{\cycleRootedTreeSize}[1]{\ensuremath{\ifthenelse{\isempty{#1}}{\partitionElement^{\cycleSymol}}{\partitionElement^{\cycleSymol}_{#1}}}}
\newcommand{\treeSizes}[1]{\ensuremath{\ifthenelse{\isempty{#1}}{\mathbf{\partitionElement}}{\mathbf{\partitionElement}_{#1}}}}

\newcommand{\numberOfCycleRootedTree}{\ensuremath{s^{\cycleSymol}}}
\newcommand{\numberOfVertexRootedTree}{\ensuremath{s^{\vertexSymbol}}}
\newcommand{\functionalGraphDecomposition}{\ensuremath{\left(\largeDisjointUnion_{i = 1}^{\numberOfVertexRootedTree} {\vertexRootedtree{i}}\right) \disjointUnion \left(\largeDisjointUnion_{i = 1}^{\numberOfCycleRootedTree} {\cycleRootedtree{i}}\right)}}

\newcommand{\generatingFunction}[2]{\ensuremath{A_{#1}\ifthenelse{\isempty{#2}}{}{\pars*{#2}}}}
\newcommand{\generatingFunctionB}[2]{\ensuremath{B_{#1}\ifthenelse{\isempty{#2}}{}{\pars*{#2}}}}
\newcommand{\generatingFunctionC}[2]{\ensuremath{C_{#1}\ifthenelse{\isempty{#2}}{}{\pars*{#2}}}}
\newcommand{\variable}{\ensuremath{x}}
\newcommand{\polySmaller}{\preceq}
\newcommand{\polyGreater}{\succeq}
\newcommand{\degree}{\ensuremath{t}}

\newcommand{\siblingsLift}[2]{\mathsf{SLift}\ifthenelse{\isempty{#1}}{}{\pars*{#1, #2}}}
\newcommand{\moveLift}[2]{\mathsf{MLift}\ifthenelse{\isempty{#1}}{}{\pars*{#1, #2}}}
\newcommand{\alternativeCountingPolynomialU}[1]{\generatingFunctionB{\vertexU}{#1}}

\newcommand{\alternativeCountingPolynomialW}[1]{\generatingFunctionB{\vertexW}{#1}}
\newcommand{\differenceAtW}[1]{\Delta_{\vertexW}\ifthenelse{\isempty{#1}}{}{\pars*{#1}}}
\newcommand{\differenceGraph}[1]{\Delta_{\graphG}\ifthenelse{\isempty{#1}}{}{\pars*{#1}}}
\newcommand{\differenceVertexRootedTree}[1]{\Delta_{\treeRootedAt{}}\ifthenelse{\isempty{#1}}{}{\pars*{#1}}}
\newcommand{\differenceCycleRootedTree}[1]{\Delta_{\cycleRootedtree{}}\ifthenelse{\isempty{#1}}{}{\pars*{#1}}}
\newcommand{\positiveProportional}{\propto_+}
\newcommand{\differenceAtVertex}[2]{\Delta_{#1}\ifthenelse{\isempty{#2}}{}{\pars*{#1}}}
\newcommand{\difference}[1]{\Delta\ifthenelse{\isempty{#1}}{}{\pars*{#1}}}

\newcommand{\leafRemove}[2]{\mathsf{LeafRemove}\ifthenelse{\isempty{#1}}{}{\pars*{#1, #2}}}
\newcommand{\cycleExtend}[2]{\mathsf{CycleExtend}\ifthenelse{\isempty{#1}}{}{\pars*{#1, #2}}}

\newcommand{\cycleLengthVector}{\ensuremath{\boldsymbol\ell}}
\newcommand{\cycleLengthSpace}{\ensuremath{D}}
\newcommand{\graphIOf}[1]{\ensuremath{I\pars*{#1}}}

\newcommand{\partitionsOf}{\vdash}
\newcommand{\partitionElement}{\ensuremath{d}}

\newcommand{\cycleIndicator}[1]{\ifthenelse{\isempty{#1}}{\mathbf{c}}{c_{#1}}}
\newcommand{\parameterizedGraph}[1]{\ifthenelse{\isempty{#1}}{\graphJ}{\graphJ\pars*{#1}}}
\newcommand{\edgeCount}[1]{\ifthenelse{\isempty{#1}}{e}{e\pars*{#1}}}
\newcommand{\smallerElement}{a}
\newcommand{\largerElement}{b}
\newcommand{\maximalTreeInnerSizes}{\treeSizes{}^*}

\newcommand{\objective}[1]{\Psi\ifthenelse{\isempty{#1}}{}{\pars*{#1}}}

\newcommand{\numberOfPairs}{n}
\newcommand{\setA}{A}
\newcommand{\setB}{B}
\newcommand{\setsSize}{m}
\newcommand{\indicesSet}{Q}
\newcommand{\indicesSetSize}{q}
\newcommand{\maximalGeneartingFunction}[1]{\Psi^*\ifthenelse{\isempty{#1}}{}{\pars*{#1}}}

\newcommand{\hammingWeight}[1]{\lvert#1\rvert_H}

\newcommand{\defemph}[1]{\textbf{\emph{#1}}}

\title{Counting Successor-Closed Subsets of Functional Digraphs}
\author{Mathias Marty}
\date{\today}

\begin{document}

\maketitle

\begin{abstract}
	A functional digraph is a directed graph where each vertex has an out-degree of at most $1$.
	We study the number of \emph{successor-closed} subsets of a functional digraph, that is, subsets from which no edge leaves.
	We show that functional digraphs have a simple recursive formula for their corresponding generating function.
	Using this formula, we determine, among all functional digraphs with a fixed number of vertices and edges, the one that maximizes and the one that minimizes the number of successor-closed subsets of every size simultaneously.
	Somewhat unexpectedly, this extremal result yields a quantitative strengthening of the \emph{set-pairs inequality} of Bollob\'as: rather than merely guaranteeing that some pair of a large enough family must violate the hypothesis of the theorem, we show that a uniformly random subset of the family witnesses a violation with high probability, quantitatively in terms of how far the family size exceeds the classical threshold.
	We further show that the same approach applies to the \emph{skew variant} of Bollob\'as's inequality due to Heged\H{u}s and Frankl, yielding an analogous probabilistic strengthening.
\end{abstract}

\section{Introduction}

A functional directed graph (digraph), which we define as a graph where every vertex has an out-degree of at most one, decomposes canonically into vertex-rooted trees and cycle-rooted trees (\zcref[S]{pro:functional-graph-as-cycle-rooted-forest}).
Vertex-rooted trees are standard trees where the edges flow from the leaves to the root.
On the other hand, cycle-rooted trees are composed of a cycle, and each of its vertices is the root of a vertex-rooted tree.
In this paper, we study a natural enumerative property of functional digraphs: the number of \emph{successor-closed} subsets, that is, subsets $\subsetOfVertices$ of vertices from which no edge leaves.
The decomposition into vertex-rooted and cycle-rooted trees yields a simple recursive formula for the \emph{successor-closed counting polynomial} $\generatingFunction{\graphG}{}$, whose coefficient of $\variable^\indicesSetSize$ counts the successor-closed subsets of size $\indicesSetSize$ (\zcref[S]{thm:successor-closed-counting-polynomial-for-functional-digraph}).
We use this formula to solve an extremal problem: among all functional digraphs with a fixed number of vertices and edges, which one maximizes and which one minimizes the number of successor-closed subsets of every size simultaneously?

\begin{theorem}[Informal, see {\zcref[S]{thm:upper-bound-extremal-2}}]
	\label{thm:informal-extremal}
	If $\graphG$ is functional, has $\numberOfVertices$ vertices, and $\numberOfEdges$ edges, the number of $\indicesSetSize$-subsets of vertices without any outgoing edge is at least
	\begin{equation*}
		\binom{\numberOfVertices - \numberOfEdges}{\indicesSetSize} + \binom{\numberOfVertices - \numberOfEdges}{\indicesSetSize - \numberOfEdges}
	\end{equation*}
	and at most
	\begin{equation*}
		\binom{\numberOfVertices - \numberOfEdges - 1}{\indicesSetSize} + \binom{\numberOfVertices - 1}{\indicesSetSize - 1}\enspace.
	\end{equation*}
\end{theorem}

Somewhat unexpectedly, the upper bound yields a quantitative strengthening of the set-pairs inequality of Bollob\'as~\cite{Bollobas65}: encoding a violation of Bollob\'as's hypothesis as an edge on the index set $[\numberOfPairs]$ gives a digraph with a known number $\numberOfEdges$ of \emph{active vertices} --- vertices with outgoing degree at least one --- which easily reduces to a functional digraph with $\numberOfEdges$ edges, to which \zcref[S]{thm:informal-extremal} directly applies.

\begin{theorem}[Informal, see {\zcref[S]{thm:probabilistic-combinatorial-lemma}}]
	\label{thm:informal-probabilistic}
	Let $\setA_1, \dotsc, \setA_\numberOfPairs, \setB_1, \dotsc, \setB_\numberOfPairs$ be set-pairs satisfying $\card{\setA_i} + \card{\setB_i} \leq \setsSize$.
	Then, if $\numberOfPairs > \binom{\setsSize}{\setsSize/2}$, we have
	\begin{equation*}
		\Pr_{\indicesSet \leftarrow \binom{[\numberOfPairs]}{\indicesSetSize}}
		\left[ \exists i \in \indicesSet \quad \exists j \in [\numberOfPairs] \setminus \indicesSet \quad \setA_i \cap \setB_j \subseteq \setB_i\right]
		\geq 1 - \exp\!\left(-\frac{\indicesSetSize(\numberOfEdges+1)}{\numberOfPairs}\right) - \frac{\indicesSetSize}{\numberOfPairs}\enspace,
	\end{equation*}
	where $\numberOfEdges \defeq \numberOfPairs - \binom{\setsSize}{\setsSize/2}$.
\end{theorem}

\subsection{Related work}

Bollob\'as's set-pairs inequality~\cite{Bollobas65}, also known as the two families theorem, has been generalized along several axes since its original publication.
Frankl~\cite{Frankl83} and, independently, Kalai~\cite{Kalai84} observed that Lov\'asz's linear-algebraic proof of the uniform case extends to a skew (asymmetric) hypothesis; the skew bound remained restricted to the uniform setting until it was extended to the fully non-uniform (weighted) case by Heged\H{u}s~\cite{Hegedus23} and by Heged\H{u}s and Frankl~\cite{HegedusF24}.
In another direction, Talbot~\cite{Talbot04} extended the theorem to $t$-intersecting families, and O'Neill and Verstra\"ete~\cite{ONeill21} generalized the inequality to $k$-wise intersection conditions among $k$ families, connecting it to a hypergraph covering problem.
Our contribution is orthogonal to these structural generalizations: rather than weakening or reshaping the intersection hypothesis, we strengthen the conclusion from an existence statement to a \emph{probabilistic guarantee} for a fixed hypothesis.
In \zcref[S]{sec:probabilistic-bollobas}, we show this strengthening applies equally to the skew variant of Heged\H{u}s and Frankl as to the standard one.

\subsection{Organization}

\zcref[S]{sec:preliminaries} recalls the necessary background on functional digraphs and the coefficient-wise order on generating functions.
\zcref[S]{sec:counting-polynomial} establishes the recursive formula for the successor-closed counting polynomial and the closed formula for elementary graphs.
\zcref[S]{sec:stars-and-concentrated-cycle-rooted-stars} reduces an arbitrary functional digraph to a canonical form built from these elementary graphs, while preserving the coefficient-wise order.
\zcref[S]{sec:extremal-digraph} refines this reduction under a fixed edge and vertex count, proving \zcref[S]{thm:informal-extremal}.
\zcref[S]{sec:probabilistic-bollobas} proves the probabilistic extension of set-pairs inequality (\zcref[S]{thm:informal-probabilistic}), and also derives the skew analog as a corollary.

\section{Preliminaries}
\label{sec:preliminaries}

We denote vectors with bold letters, and for a vector $\treeSizes{}$, we let $\card*{\treeSizes{}}$ denote its length. 
Given any integer $n$, we write $[n] \defeq \set*{1, \dotsc, n}$ and symbolize by $\treeSizes{} \partitionsOf n$ the fact that $\treeSizes{}$ is an integer partition of $n$.
In the following, we denote by $\binom{S}{n}$ the set of subsets of $S$ of size $n$. 
Moreover, given a set $S$ and an integer $n \leq \card*{S}$, we use $\indicesSet \leftarrow \binom{S}{n}$ for the random variable $\indicesSet$ uniformly distributed in $\binom{S}{n}$.
Given two polynomials $\generatingFunction{}{}$ and $\generatingFunctionB{}{}$, we write $\generatingFunction{} {}\positiveProportional \generatingFunctionB{}{}$ if there exists another polynomial $\generatingFunctionC{}{}\polyGreater 0$ such that $\generatingFunction{}{} = \generatingFunctionC{}{} \cdot \generatingFunctionB{}{}$.
Note that if $\generatingFunctionB{}{} \polyGreater 0$, then $\generatingFunction{}{} \polyGreater 0$, since $\generatingFunction{}{}$ is then a product of two polynomials with nonnegative coefficients.
Given a function $f$ and a substitution of one or more of its arguments, we write $f\big|_{x \mapsto v}$ for $f$ evaluated with $x$ replaced by $v$, the remaining arguments unchanged.

\subsection{Graph theory}

Throughout this paper, we only consider directed graphs, or \emph{digraphs} for short, without self-loops.
Given a digraph $\graphG \defeq \pars*{\vertices, \edges}$, we denote by $\graphG(\vertexU)$ the set of vertices that can be reached in one step from $\vertexU$. 
For two  graphs $\graphG$ and $\graphH$, we denote by $\graphG \disjointUnion \graphH$ their \emph{disjoint union}.
Finally, we say that a vertex is active if it has a non-zero outgoing degree.

\begin{definition}
	We say that a digraph $\graphG$ is \defemph{functional} if the outgoing degree of its vertices is at most one.
\end{definition}
\noindent The term functional refers to the fact that we can view a functional graph $\graphG$ as a function from its active vertices to $\vertices$.
When $\graphG$ is functional, instead of defining $\graphG\pars*{\cdot}$ to return a singleton containing a vertex, we define it to return the vertex directly.

We continue by introducing some special graphs.
\begin{definition}
	A \defemph{$\cycleLength{}$-cycle}, or $\cycle{\cycleLength{}}$ in symbol, is a digraph with $\cycleLength{}$ vertices $\vertexU_1, \dotsc, \vertexU_{\cycleLength{}}$ and the structure $\connect{\vertexU_1}{\vertexU_2}, \dotsc, \connect{\vertexU_{\cycleLength{} - 1}}{\vertexU_{\cycleLength{}}}, \connect{\vertexU_{\cycleLength{}}}{\vertexU_1}$.
	The integer $\cycleLength{}$ is called the \defemph{cycle length}.
\end{definition}
\begin{definition}
	A \defemph{vertex-rooted tree} is an acyclic digraph with a distinguished vertex called its \defemph{root}.
	Edges are flowing from the leaf vertices to the root vertex.
\end{definition}
\begin{definition}
	We say that a vertex-rooted tree is a \defemph{vertex-rooted star} when it has a depth of at most one.
	More precisely, we denote by $\starGraph{\treeInternalSize{}}$ the star which has $\treeInternalSize{}$ leaves. 
\end{definition}
\begin{definition}
	A \defemph{vertex-rooted path} is a vertex-rooted tree where the tree is degenerate.
	That is, the tree is a path containing a finite number of vertices $\treeSize{}$.
	The length of the path, denoted by $\pathLength$, is the number of edges, that is, $\pathLength = \treeSize{} - 1$.
	We use the notation $\pathGraph{\pathLength}$ for a vertex-rooted path of length $\pathLength$.
\end{definition}

Instead of a vertex root, we now consider graphs whose root is a cycle and each cycle element is a vertex-rooted elementary graph.
\begin{definition}
	A \defemph{directed cycle-rooted tree}, or \defemph{cycle-rooted tree} for short, is a cycle where the vertices of the cycle are each a root of their own vertex-rooted trees.
	Similarly, a \defemph{cycle-rooted star} is a cycle-rooted tree whose vertex-rooted trees have depth one.
	In symbols, we will use $\cycleRootedStar{\treeInternalSizes{}}$ for the cycle-rooted star whose cycle has length $\cycleLength{} \defeq \card*{\treeInternalSizes{}}$, where the $i$-th cycle vertex is the root of a tree with $\treeInternalSize{i}$ leaves.
	Finally, a \defemph{cycle-rooted path} is a cycle-rooted tree whose vertex-rooted trees are all paths.
	We denote by $\cycleRootedPath{\pathLengths}$ the cycle-rooted path whose cycle has length $\card*{\pathLengths}$ and the $i$-th path has length $\pathLength_i$.
\end{definition}
\noindent A comparison between a vertex-rooted and cycle-rooted star is depicted in \zcref[S]{fig:tree-cycle-rooted-tree}.
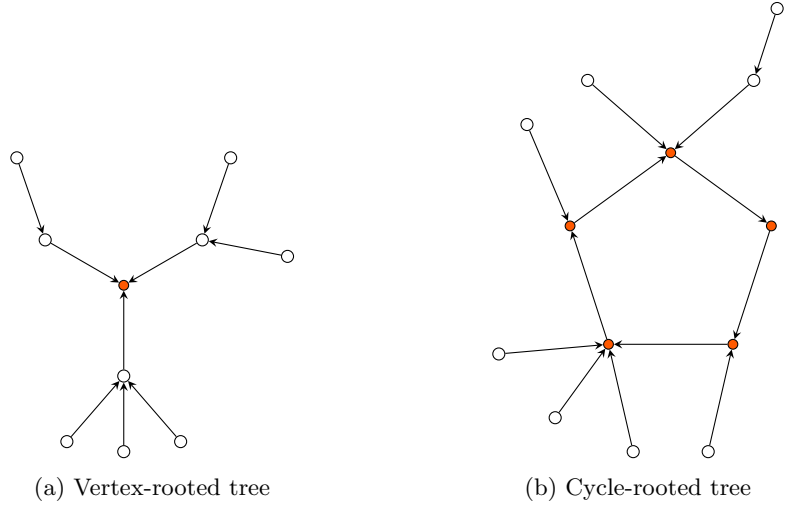
\begin{figure}[t]
	\centering
	\begin{subfigure}[b]{0.42\textwidth}
		\centering
		\begin{tikzpicture}[>=stealth]
			\vertexrootedtree
		\end{tikzpicture}
		\caption{Vertex-rooted tree}
	\end{subfigure}
	\hfill
	\begin{subfigure}[b]{0.52\textwidth}
		\centering
		\begin{tikzpicture}[>=stealth]
			\cyclerootedtree
		\end{tikzpicture}
		\caption{Cycle-rooted tree}
	\end{subfigure}
	\caption{The two functional digraphs presented: (a) a vertex-rooted tree, and (b) a cycle-rooted tree. The highlighted vertices mark the vertex-root (a) and the cycle-root (b).}
	\label{fig:tree-cycle-rooted-tree}
\end{figure}

\begin{remark}
	The graph $\starGraph{0}$ consists of a single vertex.
	Throughout this paper, we identify isolated vertices with $\starGraph{0}$.
\end{remark}

When all the weight of a cycle-rooted graph is concentrated in a single slot, we say that the graph is \emph{concentrated}.
\begin{definition}
	A \defemph{concentrated cycle-rooted tree} is a cycle-rooted tree where only one of the cycle vertices is not a bare vertex.
	Similarly, a \defemph{concentrated cycle-rooted star} $\cycleRootedStar{\treeInternalSizes{}}$ is a cycle-rooted star where, $\treeInternalSizes{} = \pars*{\treeInternalSize{}, 0, \dotsc, 0}$ and a \defemph{concentrated cycle-rooted path} $\cycleRootedPath{\pathLengths}$ is a cycle-rooted path where, $\pathLengths = \pars*{\pathLength, 0, \dotsc, 0}$.
	We use the notations $\concentratedCycleRootedStar{\cycleLength{}}{\treeInternalSize{}} \defeq \cycleRootedStar{\treeInternalSizes{}}$ for the concentrated cycle-rooted star with a cycle of length $\cycleLength{}$ and $\concentratedCycleRootedPath{\cycleLength{}}{\pathLength} \defeq \cycleRootedPath{\pathLengths}$ for the cycle-rooted path with cycle length $\cycleLength{}$.
\end{definition}

Finally, we define the \emph{reverse} operator for graphs.
\begin{definition}
	Given any graph $\graphG$, we denote by $\rev{\graphG}$ the graph in which all the edges are reversed.
	We call this graph the \defemph{reversed $\graphG$}.
\end{definition}

We continue with a characterization result of functional digraphs.
\begin{proposition}[Characterization of functional digraph]\label{pro:functional-graph-as-cycle-rooted-forest}
	Every functional digraph $\graphG$ is a disjoint union of cycle-rooted and vertex-rooted trees.
	In other words, it can be decomposed into
	\begin{equation*}
		\graphG = \functionalGraphDecomposition\enspace,
	\end{equation*}
	where the $\vertexRootedtree{i}$'s and the $\cycleRootedtree{i}$'s are, respectively, vertex-rooted trees and cycle-rooted trees.
	We will denote by $\vertexRootedTreeSize{i}$ and $\cycleRootedTreeSize{i}$ their respective sizes.
	Moreover, we also denote by $\cycleLength{i}$ the cycle length of $\cycleRootedtree{i}$.
\end{proposition}
\begin{proof}
	Since each vertex has out-degree at most one, following out-edges from any vertex either reaches a fixed point-free cycle or terminates at a sink, splitting $\graphG$ into connected components that are each either cycle-rooted or vertex-rooted.
\end{proof}

\subsection{Generating functions and successor-closed subsets}

We define the following family:
\begin{definition}
	The set
	\begin{equation*}
		\successorClosedSubsets{\graphG} \defeq \set*{\indicesSet \subseteq \vertices \mid \forall \vertexU \in \indicesSet,\ \graphG(\vertexU) \subseteq \indicesSet}
	\end{equation*}
	is the \defemph{family of successor-closed subsets of $\graphG$}.
	We say that a subset $\indicesSet$ of vertices is \defemph{successor-closed} if and only if it belongs to $\successorClosedSubsets{\graphG}$.
\end{definition}

Throughout this paper, we will only consider \emph{univariate polynomials}.

\begin{notation}
	Given a polynomial $\generatingFunction{}{}$, we denote by $\coeff*{\variable^\degree} \generatingFunction{}{\variable}$ the coefficient of the monomial $\variable^\degree$ of $\generatingFunction{}{}$.
\end{notation}

We define the following order for polynomials.

\begin{definition}
	Given two polynomials $\generatingFunction{}{}$ and $\generatingFunctionB{}{}$, the \defemph{coefficient-wise order} is defined as 
	\begin{equation*}
		\generatingFunction{}{} \polySmaller \generatingFunctionB{}{} \Leftrightarrow \forall \degree \in \N \quad \coeff*{\variable^\degree}\generatingFunction{}{} \leq \coeff*{\variable^\degree}\generatingFunctionB{}{}\enspace.
	\end{equation*}
	We will also write $\generatingFunctionB{}{} \polyGreater \generatingFunction{}{}$ when $\generatingFunction{}{} \polySmaller \generatingFunctionB{}{}$.
\end{definition}
\noindent The next result follows from basic algebra.
\begin{proposition}
	The relation $\polySmaller$ is a partial order over $\R[\variable]$.
\end{proposition}
\noindent We will implicitly use the transitivity property of this relation throughout this paper.

We capture our counting problem in the following generating function.
\begin{definition}
	The polynomial
	\begin{equation*}
		\generatingFunction{\graphG}{\variable} \defeq \sum_{\indicesSet \in \successorClosedSubsets{\graphG}} \variable^{\card*{\indicesSet}}
	\end{equation*}
	is the \defemph{successor-closed counting polynomial} of the digraph $\graphG$.
\end{definition}
\noindent The coefficients of this polynomial count how many successor-closed subsets of a particular cardinality exist.
With our notations,
\begin{equation*}
	\coeff*{\variable^\indicesSetSize} \generatingFunction{\graphG}{\variable} = \card*{\set*{\indicesSet \in \successorClosedSubsets{\graphG} \mid \card*{\indicesSet} = \indicesSetSize}}\enspace.
\end{equation*}

\begin{remark}
	\label{rmk:star-and-concentrated-cycle-rooted-star-triviality}
	Note that $\starGraph{0}$ is a single vertex.
	Furthermore, as a self-loop does not affect the family of successor-closed subsets, we identify any cycle-rooted graph with cycle length $1$ as its vertex-rooted version.
	In particular, $\generatingFunction{\concentratedCycleRootedStar{1}{\treeInternalSize{}}}{} = \generatingFunction{\starGraph{\treeInternalSize{}}}{}$ and $\generatingFunction{\concentratedCycleRootedPath{1}{\pathLength{}}}{} = \generatingFunction{\pathGraph{\pathLength{}}}{}$.
\end{remark}

Given the disjoint union of two graphs, its successor-closed subsets are the product of the successor-closed subsets of the two graphs.
\begin{lemma}[Product rule]
	\label{lma:product-rule}
	Let $\graphG$ and $\graphH$ be two digraphs, then
	\begin{equation*}
		\generatingFunction{\graphG \disjointUnion \graphH}{\variable} = \generatingFunction{\graphG}{\variable} \cdot \generatingFunction{\graphH}{\variable}\enspace.
	\end{equation*}
\end{lemma}
\begin{proof}
	The successor-closed subsets can be chosen independently in $\graphG$ and $\graphH$.
	Therefore, the successor-closed counting polynomial is simply the product of everything.
\end{proof}

A common technique throughout this paper is to find a local transformation of the graph $\graphG$ that either increases or decreases the successor-closed counting polynomial while making the graph easier to analyze.
A very simple example is edge deletion.
\begin{lemma}[Monotonicity under edge deletion]
	\label{lma:monotonicity-under-edge-deletion}
	Let $\graphG=\pars*{\vertices,\edges}$ be a digraph and let $\graphH= \pars*{\vertices,\edges\setminus{(\vertexU,\vertexV)}}$ be obtained by removing a single edge. Then
	\begin{equation*}
		\generatingFunction{\graphG}{\variable}
		\polySmaller
		\generatingFunction{\graphH}{\variable}\enspace.
	\end{equation*}
\end{lemma}
\begin{proof}
	Every successor-closed subset of $\graphG$ is also successor-closed in
	$\graphH$, since every edge of $\graphH$ is also an edge of $\graphG$.
\end{proof}

\subsection{Bollobás set-pair inequalities}

In \cite{Bollobas65}, \citeauthor{Bollobas65} introduces a fundamental result of extremal set theory.
We state the theorem formally below.

\begin{lemma}[Combinatorial lemma~{\cite{Bollobas65}}]
	\label{lma:combinatorial-lemma}
	Let $\setA_1, \dotsc, \setA_\numberOfPairs, \setB_1, \dotsc, \setB_\numberOfPairs$ be some subsets of some universe
	satisfying $\setA_i \cap \setB_i = \varnothing$ and $\card{\setA_i} + \card{\setB_i} \leq \setsSize$ for every $i$.
	Then, if for all $i \neq j$ we have $\setA_i \cap \setB_j \not\subseteq \setB_i$, it must be the case that
	$\numberOfPairs \leq \binom{\setsSize}{\setsSize/2}$.
\end{lemma}
\begin{proof}
	Since $\setA_i \cap \setB_i = \varnothing$ for every $i$, the condition $\setA_i \cap \setB_j \not\subseteq \setB_i$ is equivalent to $\setA_i \cap \setB_j \neq \varnothing$ for $i \neq j$ (if $\setA_i\cap\setB_j\subseteq\setB_i$, every element of $\setA_i\cap\setB_j$ would lie in both $\setA_i$ and $\setB_i$, contradicting their disjointness unless $\setA_i\cap\setB_j=\varnothing$, in which case the containment holds trivially).
	Bollob\'as's inequality~\cite{Bollobas65} then gives $\sum_{i=1}^{\numberOfPairs}\binom{\card{\setA_i}+\card{\setB_i}}{\card{\setA_i}}^{-1}\le1$. 
	Since $\binom{m}{k}\le\binom{m}{\lfloor m/2\rfloor}$, for every integer $k$, and $m\mapsto\binom{m}{\lfloor m/2\rfloor}$ is increasing, every summand is at least $\binom{\setsSize}{\setsSize/2}^{-1}$, so $\numberOfPairs \le \binom{\setsSize}{\setsSize/2}$.
\end{proof}

Recently, \citeauthor{HegedusF24} introduced a variant of this lemma.
Their \emph{skew} variant relaxes the hypothesis so that it only needs to be satisfied over indices $i < j$ rather than $i \neq j$.
\begin{lemma}[Skew combinatorial lemma~{\cite{HegedusF24}}]
	\label{lma:skew-combinatorial-lemma}
	Let $\setA_1, \dotsc, \setA_\numberOfPairs, \setB_1, \dotsc, \setB_\numberOfPairs$ be some subsets of the universe $[\setsSize]$, for some $\setsSize \in \N$, which satisfy $\setA_i \cap \setB_i = \varnothing$ for every $i$.
	Then, if for all $i < j$ we have $\setA_i \cap \setB_j \not\subseteq \setB_i$, it must be the case that $\numberOfPairs \leq \pars*{\setsSize+1}\binom{\setsSize}{\setsSize/2}$.
\end{lemma}
\begin{proof}
	As in the proof of \zcref[S]{lma:combinatorial-lemma}, disjointness of $\setA_i,\setB_i$ makes $\setA_i\cap\setB_j\not\subseteq\setB_i$ equivalent to $\setA_i\cap\setB_j\neq\varnothing$.
	By~\cite{HegedusF24}, $\sum_{i=1}^{\numberOfPairs}\binom{\card{\setA_i}+\card{\setB_i}}{\card{\setA_i}}^{-1}\le \setsSize+1$, and since every summand is at least $\binom{\setsSize}{\setsSize/2}^{-1}$, we get $\numberOfPairs \le \pars*{\setsSize+1}\binom{\setsSize}{\setsSize/2}$.
\end{proof}

\section{The successor-closed counting polynomial of functional digraphs}
\label{sec:counting-polynomial}

We first consider a vertex-rooted tree $\vertexRootedtree{}$ and for $\vertexU \in \vertexRootedtree{}$ we define  $\treeRootedAt{\vertexU}$ to be the vertex-rooted subtree of $\vertexRootedtree{}$ rooted at $\vertexU$.
Then, we define
\begin{equation}\label{equ:successor-closed-counting-polynomial-for-tree}
	\generatingFunction{{\treeRootedAt{\vertexU}}}{\variable}\defeq
	\begin{cases}
		1 + \variable, & \text{if $\vertexU$ is a leaf,}\\
		1 + \variable\prod_{\vertexV \in \rev{\graphG}(\vertexU)} \generatingFunction{{\treeRootedAt{\vertexV}}}{\variable}, & \text{otherwise.}
	\end{cases}
\end{equation}
Building on this definition, we define, for every cycle-rooted tree $\cycleRootedtree{}$ whose root is a directed cycle $\cycle{}$ containing vertices $\vertexU_1, \dotsc, \vertexU_{\cycleLength{}}$, the polynomial
\begin{equation}\label{equ:successor-closed-counting-polynomial-for-cycle-rooted-tree}
	\generatingFunction{{\cycleRootedtree{}}}{\variable} \defeq 1 + \prod_{i = 1}^{\cycleLength{}} \left(\generatingFunction{{\treeRootedAt{\vertexU_i}}}{\variable} - 1\right)\enspace,
\end{equation}
where $\treeRootedAt{\vertexU_i}$ is the vertex-rooted tree rooted at $\vertexU_i$, excluding the vertices in the cycle.
\noindent We later prove that these two expressions are the correct generating functions for the vertex-rooted and cycle-rooted trees.

\begin{theorem}\label{thm:successor-closed-counting-polynomial-for-functional-digraph}
	For every functional digraph $\graphG \defeq \functionalGraphDecomposition$, its successor-closed counting polynomial is
	\begin{equation*}
		\generatingFunction{\graphG}{\variable} = \left(\prod_{i = 1}^{\numberOfVertexRootedTree} \generatingFunction{{\vertexRootedtree{i}}}{\variable}\right) \cdot \left(\prod_{i = 1}^{\numberOfCycleRootedTree} \generatingFunction{{\cycleRootedtree{i}}}{\variable}\right)\enspace.
	\end{equation*}
\end{theorem}
\noindent Before presenting the proof, we present some useful special cases of this theorem.

\begin{corollary}\label{cly:generating-function-cycle}
	The cycle $\cycle{\cycleLength{}}$ has successor-closed counting polynomial:
	\begin{equation*}
		\generatingFunction{\cycle{\cycleLength{}}}{\variable} = 1 + \variable^{\cycleLength{}}\enspace.
	\end{equation*}
\end{corollary}

\begin{corollary}\label{cly:generating-function-path}
	The path $\pathGraph{\pathLength}$ has successor-closed counting polynomial:
	\begin{equation*}
		\generatingFunction{\pathGraph{\pathLength}}{\variable} = \sum_{i=0}^{\pathLength + 1} \variable^i\enspace.
	\end{equation*}
\end{corollary}

\begin{corollary}\label{cly:generating-function-vertex-rooted-star}
	The vertex-rooted star $\starGraph{\treeInternalSize{}}$ has successor-closed counting polynomial:
	\begin{equation*}
		\generatingFunction{\starGraph{\treeInternalSize{}}}{\variable} = 1 + \variable\pars*{1 + \variable}^{\treeInternalSize{}}\enspace.
	\end{equation*}
\end{corollary}

\begin{corollary}\label{cly:generating-function-concentrated-cycle-rooted-star}
	The concentrated cycle-rooted star $\concentratedCycleRootedStar{\cycleLength{}}{\treeInternalSize{}}$ has successor-closed counting polynomial:
	\begin{equation*}
		\generatingFunction{\concentratedCycleRootedStar{\cycleLength{}}{\treeInternalSize{}}}{\variable} = 1 + \variable^{\cycleLength{}} \pars*{1 + \variable}^{\treeInternalSize{}}\enspace.
	\end{equation*}
\end{corollary}

\begin{proof}[Proof of {\zcref[S]{thm:successor-closed-counting-polynomial-for-functional-digraph}}]
	We first analyze the case of a vertex-rooted tree and then continue with a cycle-rooted tree.

	\paragraph{The case of vertex-rooted trees} 
	We show that \zcref[S]{equ:successor-closed-counting-polynomial-for-tree} is the correct generating function for a vertex-rooted tree.
	At $\vertexU$, there are two possibilities for a successor-closed subset: either $\vertexU$ is not selected, in which case no vertex in any child subtree can be selected, or $\vertexU$ is selected, in which case we may independently choose a successor-closed subset from each child subtree.
	From this observation, we directly obtain \zcref[S]{equ:successor-closed-counting-polynomial-for-tree}.
	For a leaf $\vertexU$, the generating function is $\generatingFunction{{\treeRootedAt{\vertexU}}}{\variable} = 1 + \variable$ as there are two successor-closed subsets: the empty set and the singleton containing $\vertexU$.

	\paragraph{The case of cycle-rooted trees} 
	We now consider a cycle-rooted tree $\cycleRootedtree{}$ whose root is a directed cycle $\cycle{}$ containing vertices $\vertexU_1, \dotsc, \vertexU_{\cycleLength{}}$.
	A successor-closed subset either contains nothing or a product of non-empty successor-closed subsets of the child subtrees and the complete cycle.
	Therefore, the generating function is
	\begin{align*}
		\generatingFunction{{\cycleRootedtree{}}}{\variable} &= 1 + \variable^{\cycleLength{}} \prod_{i = 1}^{\cycleLength{}} \frac{\generatingFunction{{\treeRootedAt{\vertexU_i}}}{\variable} - 1}{\variable}\\
		&=1 + \prod_{i = 1}^{\cycleLength{}} \left(\generatingFunction{{\treeRootedAt{\vertexU_i}}}{\variable} - 1 \right)\enspace,
	\end{align*}
	where the division by $\variable$ is here to prevent double counting the root of the chosen subtrees and the minus one to remove the empty selection. This is exactly \zcref[S]{equ:successor-closed-counting-polynomial-for-cycle-rooted-tree}.

	Finally, since a functional digraph $\graphG$ is a disjoint union of trees, \zcref[S]{lma:product-rule} implies that the successor-closed counting polynomial is simply the product of everything.
\end{proof}

We conclude this section with two useful inequalities.

\begin{definition}
	Let $\vertexRootedtree{}$ be a vertex-rooted tree and $\vertexV \neq \vertexW$ be a leaf different from the root.
	The \defemph{leaf removal at $\vertexV$}, denoted by $\leafRemove{\vertexRootedtree{}}{\vertexV}$, is the tree obtained from $\vertexRootedtree{}$ by deleting $\vertexV$ and its incident edge.
\end{definition}

\begin{lemma}[Removing a leaf]\label{lma:removing-a-leaf}
	For every vertex-rooted tree $\vertexRootedtree{}$ and leaf $\vertexV$ different from its root, letting $\widehat{\vertexRootedtree{}} \defeq \leafRemove{\vertexRootedtree{}}{\vertexV}$, it holds that
	\begin{equation*}
		\generatingFunction{\vertexRootedtree{}}{} - \variable \generatingFunction{\widehat{\vertexRootedtree{}}}{} \polyGreater 1\enspace.
	\end{equation*}
\end{lemma}
\begin{proof}
	We prove this theorem by induction on the depth of $\vertexV$ in the tree.
	Let $\vertexW$ be the root of the tree.
	
	\paragraph{Base case ($\vertexV$ is a direct child of $\vertexW$)}
	Let $\vertexU_1, \dotsc, \vertexU_q$ be the other children of the root $\vertexW$ and $\alternativeCountingPolynomialW{}  \polyGreater 1$ be the product of their generating functions, that is $\alternativeCountingPolynomialW{\variable} \defeq \prod_{i =1}^q \generatingFunction{{\treeRootedAt{\vertexU_i}}}{\variable}$.
	From \zcref[S]{thm:successor-closed-counting-polynomial-for-functional-digraph}, we get the following two polynomials:
	\begin{align*}
		\generatingFunction{\vertexRootedtree{}}{}&= 1 + \variable (1 + \variable) \alternativeCountingPolynomialW{\variable}\\
		&=1 +  \variable \alternativeCountingPolynomialW{\variable}  + \variable^2 \alternativeCountingPolynomialW{\variable} \enspace, \\
		\generatingFunction{\widehat{\vertexRootedtree{}}}{} &= 1 + \variable \alternativeCountingPolynomialW{\variable}
	\end{align*}
	and the difference is 
	\begin{equation*}
		\generatingFunction{\vertexRootedtree{}}{} - \variable \generatingFunction{\widehat{\vertexRootedtree{}}}{} = 1 + \variable\pars*{\alternativeCountingPolynomialW{\variable} - 1} \polyGreater 1\enspace,
	\end{equation*}
	since $\alternativeCountingPolynomialW{} \polyGreater 1$.
	
	\paragraph{Inductive step ($\vertexV$ is not a direct child of $\vertexW$)}
	Let $\vertexV'$ be the only successor of $\vertexV$ right below $\vertexW$, $\vertexU_1, \dotsc, \vertexU_q$ be the other children of the root $\vertexW$, and $\alternativeCountingPolynomialW{} \polyGreater 1$ be, as before, the product of their generating functions.
	From \zcref[S]{thm:successor-closed-counting-polynomial-for-functional-digraph}, we have:
	\begin{align*}
		\generatingFunction{\vertexRootedtree{}}{}&= 1 + \variable \alternativeCountingPolynomialW{\variable} \cdot \generatingFunction{\treeRootedAt{\vertexV'}}{\variable} \enspace, \\
		\generatingFunction{\widehat{\vertexRootedtree{}}}{} &= 1 + \variable \alternativeCountingPolynomialW{\variable} \cdot \generatingFunction{\widehat{\treeRootedAt{\vertexV'}}}{\variable}\enspace, 
	\end{align*}
	where $\widehat{\treeRootedAt{\vertexV'}} \defeq \leafRemove{\treeRootedAt{\vertexV'} }{\vertexV}$.
	The difference is
	\begin{align*}
		\generatingFunction{\vertexRootedtree{}}{} - \variable \generatingFunction{\widehat{\vertexRootedtree{}}}{} &= 1 + \variable \pars*{\alternativeCountingPolynomialW{\variable}\cdot \generatingFunction{\treeRootedAt{\vertexV'}}{\variable}  - \variable\alternativeCountingPolynomialW{\variable} \cdot \generatingFunction{\widehat{\treeRootedAt{\vertexV'}}}{\variable} - 1} \\
		&=1 + \variable \pars*{\alternativeCountingPolynomialW{\variable}\cdot \pars*{ \generatingFunction{\treeRootedAt{\vertexV'}}{\variable}  - \variable \cdot \generatingFunction{\widehat{\treeRootedAt{\vertexV'}}}{\variable}}- 1}\\
		&\polyGreater 1\enspace,
	\end{align*}
	where the last inequality follows from the induction hypothesis.
\end{proof}

\begin{definition}
	Let $\cycleRootedtree{}$ be a cycle-rooted tree and $\vertexV$ be a leaf of one of its vertex-rooted subtrees, rooted at $\vertexW$.
	The \defemph{cycle extension at $\vertexV$}, denoted by $\cycleExtend{\cycleRootedtree{}}{\vertexV}$, is the cycle-rooted tree obtained from $\cycleRootedtree{}$ by replacing the subtree rooted at $\vertexW$ with $\leafRemove{\vertexRootedtree{\vertexW}}{\vertexV}$, and inserting a new bare vertex into the cycle, extending the cycle length by one.
	Note that the number of vertices and edges is preserved.
\end{definition}

\begin{lemma}[Extending a cycle]\label{lma:leaf-to-cycle}
	For every cycle-rooted tree $\cycleRootedtree{}$ and leaf $\vertexV$ of one of its vertex-rooted subtrees, letting $\widehat{\cycleRootedtree{}} \defeq \cycleExtend{\cycleRootedtree{}}{\vertexV}$, it holds that
	\begin{equation*}
		\generatingFunction{\cycleRootedtree{}}{}
		\polyGreater
		\generatingFunction{\widehat{\cycleRootedtree{}}}{}\enspace.
	\end{equation*}
\end{lemma}
\begin{proof}
	Let $\vertexW$ be the root of the affected vertex-rooted tree and 
	\begin{equation*}
		\widehat{\vertexRootedtree{\vertexW}} \defeq \leafRemove{\vertexRootedtree{\vertexW}}{\vertexV}
	\end{equation*}
	be the new tree.
	By \zcref[S]{thm:successor-closed-counting-polynomial-for-functional-digraph} and by factorizing the common terms, the difference is
	\begin{equation*}
		\generatingFunction{\cycleRootedtree{}}{} - \generatingFunction{\widehat{\cycleRootedtree{}}}{} \positiveProportional \pars*{\vertexRootedtree{\vertexV} - 1} - \variable\pars*{\widehat{\vertexRootedtree{\vertexV}} - 1} = \vertexRootedtree{\vertexV} - \variable \widehat{\vertexRootedtree{\vertexV}} + \variable - 1 \polyGreater 0\enspace,
	\end{equation*}
	where the last inequality follows from \zcref[S]{lma:removing-a-leaf}.
\end{proof}

\section{From functional digraphs to stars and concentrated cycle-rooted stars}
\label{sec:stars-and-concentrated-cycle-rooted-stars}

In this section, we first introduce several transformations that yield closed-form lower and upper bounds.
The resulting theorem is stated below.
\begin{theorem}\label{thm:upper-bound-extremal}
	Let $\graphG \defeq \functionalGraphDecomposition$ be any functional digraph, 
	\begin{equation*}
		\graphF \defeq \pars*{\largeDisjointUnion_{i = 1}^{\numberOfVertexRootedTree} {\pathGraph{\vertexRootedTreeSize{i} - 1}}} \disjointUnion \pars*{\largeDisjointUnion_{i = 1}^{\numberOfCycleRootedTree} {\cycle{\cycleRootedTreeSize{i}}}}\enspace,
	\end{equation*}
	and
	\begin{equation*}
		\graphH \defeq \pars*{\largeDisjointUnion_{i = 1}^{\numberOfVertexRootedTree} {\starGraph{\vertexRootedTreeSize{i} - 1}}} \disjointUnion \pars*{\largeDisjointUnion_{i = 1}^{\numberOfCycleRootedTree} {\concentratedCycleRootedStar{\cycleLength{i}}{\cycleRootedTreeSize{i} - \cycleLength{i}}}}\enspace.
	\end{equation*}
	Then, the inequalities $\generatingFunction{\graphF}{} \polySmaller \generatingFunction{\graphG}{} \polySmaller \generatingFunction{\graphH}{}$ holds.
	Note that
	\begin{equation*}
		\generatingFunction{\graphF}{\variable} = \prod_{i = 1}^{\numberOfVertexRootedTree}\pars*{\sum_{i=0}^{\vertexRootedTreeSize{i}} \variable^i}
		\cdot
		\prod_{i = 1}^{\numberOfCycleRootedTree}\pars*{1 + \variable^{\cycleRootedTreeSize{i}}}\enspace,
	\end{equation*}
	and
	\begin{equation*}
		\generatingFunction{\graphH}{\variable} = \prod_{i = 1}^{\numberOfVertexRootedTree}\pars*{1 + \variable \pars*{1 + \variable}^{\vertexRootedTreeSize{i} - 1}}
		\cdot
		\prod_{i = 1}^{\numberOfCycleRootedTree}\pars*{1 + \variable^{\cycleLength{i}}\pars*{1 + \variable}^{\cycleRootedTreeSize{i} - \cycleLength{i}}}\enspace.
	\end{equation*}
\end{theorem}

\begin{definition}[Space of cycle-length assignments]
	\label{def:cycle-length-space}
	Fix a functional digraph $\graphG \functionalGraphDecomposition$ and let
	\begin{equation*}
		\cycleLengthSpace \defeq \prod_{i = 1}^{\numberOfCycleRootedTree} [\cycleRootedTreeSize{i}]\enspace,
	\end{equation*}
	equipped with the coordinate-wise partial order $\cycleLengthVector \leq \cycleLengthVector'$ if and only if $\cycleLength{i} \leq \cycleLength{i}'$ for every $i \in [\numberOfCycleRootedTree]$.
	For $\cycleLengthVector \in \cycleLengthSpace$, we define
	\begin{equation*}
		\graphIOf{\cycleLengthVector} \defeq \pars*{\largeDisjointUnion_{i = 1}^{\numberOfVertexRootedTree} \starGraph{\vertexRootedTreeSize{i} - 1}} \disjointUnion \pars*{\largeDisjointUnion_{i = 1}^{\numberOfCycleRootedTree} \concentratedCycleRootedStar{\cycleLength{i}}{\cycleRootedTreeSize{i} - \cycleLength{i}}}\enspace.
	\end{equation*}
	Its generating function is
	\begin{equation*}
		\generatingFunction{\graphIOf{\cycleLengthVector}}{\variable} = \prod_{i = 1}^{\numberOfVertexRootedTree}\pars*{1 + \variable \pars*{1 + \variable}^{\vertexRootedTreeSize{i} - 1}}
		\cdot
		\prod_{i = 1}^{\numberOfCycleRootedTree}\pars*{1 + \variable^{\cycleLength{i}} \pars*{1 + \variable}^{\cycleRootedTreeSize{i} - \cycleLength{i}}}\enspace.
	\end{equation*}
\end{definition}

\noindent The map $\cycleLengthVector \mapsto \graphIOf{\cycleLengthVector}$ is a bijection between $\cycleLengthSpace$ and the set of functional digraphs whose vertex-rooted components are the stars $\starGraph{\vertexRootedTreeSize{i}-1}$ and whose $i$-th cycle-rooted component is the concentrated cycle-rooted star $\concentratedCycleRootedStar{\cycleLength{i}}{\cycleRootedTreeSize{i}-\cycleLength{i}}$, on the same underlying component sizes (recall $\concentratedCycleRootedStar{1}{\treeInternalSize{}} \defeq \starGraph{\treeInternalSize{}}$). Through this bijection, $\pars*{\cycleLengthSpace, \leq}$ is isomorphic to this space of graphs, ordered by the cycle length of each component.

\begin{proposition}[Monotonicity and edge count on $\cycleLengthSpace$]
	\label{prop:cycle-length-space-monotone}
	The map $\cycleLengthVector \mapsto \generatingFunction{\graphIOf{\cycleLengthVector}}{}$ is coordinate-wise antitone on $\pars*{\cycleLengthSpace, \leq}$ with respect to the coefficient-wise order: for every $i$ and every $1 \leq t < t' \leq \cycleRootedTreeSize{i}$,
	\begin{equation*}
		\graphIOf{\cycleLengthVector} \big|_{\cycleLength{i} \mapsto t} \;\polyGreater\; \graphIOf{\cycleLengthVector} \big|_{\cycleLength{i} \mapsto t'}\enspace.
	\end{equation*}
	Consequently,
	\begin{equation*}
		\generatingFunction{\graphIOf{\pars*{1,\dotsc,1}}}{} \;\polyGreater\; \generatingFunction{\graphIOf{\cycleLengthVector}}{} \;\polyGreater\; \generatingFunction{\graphIOf{\pars*{\cycleRootedTreeSize{1}, \dotsc, \cycleRootedTreeSize{\numberOfCycleRootedTree}}}}{}
	\end{equation*}
	for every $\cycleLengthVector \in \cycleLengthSpace$.
	Moreover, the number of edges removed relative to the all-genuine-cycle baseline $\graphIOf{\pars*{\cycleRootedTreeSize{1}, \dotsc, \cycleRootedTreeSize{\numberOfCycleRootedTree}}}$ equals $\numberOfCycleRootedTree - \hammingWeight{\cycleLengthVector - \mathbf 1}$, where $\hammingWeight{\cdot}$ denotes the Hamming weight.
	In other words, the number of edges removed is the number of coordinates at which $\cycleLengthVector$ sits at the bottom of $\cycleLengthSpace$.
\end{proposition}

\begin{proof}
	Single-coordinate decrement follows from \zcref[S]{lma:leaf-to-cycle} applied to cycle-rooted stars.
	Then, noticing that any $\cycleLengthVector$ can be reached with a finite number of single-coordinate decrements, we obtain the bound by transitivity.
	As $\cycleLengthSpace$ is a product of totally ordered sets and $\generatingFunction{\graphIOf{\cdot}}{}$ is antitone in each coordinate, it is antitone on the product order, giving the bound at the two corners.
	For the edge count: a component with $\cycleLength{i} \geq 2$ is a genuine cycle, contributing $\cycleRootedTreeSize{i}$ edges; a component with $\cycleLength{i} = 1$ is (by \zcref[S]{rmk:star-and-concentrated-cycle-rooted-star-triviality}) a star, contributing $\cycleRootedTreeSize{i}-1$ edges, one fewer than the baseline. Summing over $i$ gives the stated formula.
	The generating function is obtained through \zcref[S]{cly:generating-function-vertex-rooted-star,cly:generating-function-concentrated-cycle-rooted-star}.
\end{proof}

Instead of proving the main theorem directly, we focus on two local moves and show that they only increase the successor-closed counting polynomial.
\begin{definition}
    Let $\treeRootedAt{}$ be a vertex-rooted tree, and let $\vertexV$ be a non-root vertex with parent $\vertexW$.
    The \defemph{sibling lift at $\vertexV$}, denoted by $\siblingsLift{{\treeRootedAt{}}}{\vertexV}$, is obtained by moving the children of $\vertexV$ from $\vertexV$ to its parent $\vertexW$: for every child $\vertexU$ of $\vertexV$, the edge $\connect{\vertexU}{\vertexV}$ is replaced by $\connect{\vertexU}{\vertexW}$.
\end{definition}
\noindent We depict the sibling lift operator in \zcref[S]{fig:sibling-lift}.

\begin{figure}[t]
	\centering
	\begin{subfigure}[b]{0.42\textwidth}
		\centering
		\begin{tikzpicture}[>=stealth]
			\siblingLiftCommon
			\draw[->] (u1) -- (v);
			\draw[->] (u2) -- (v);
			\draw[->] (u3) -- (v);
		\end{tikzpicture}
		\caption{Before: $\treeRootedAt{}$}
	\end{subfigure}
	\hfill
	\begin{subfigure}[b]{0.42\textwidth}
		\centering
		\begin{tikzpicture}[>=stealth]
			\siblingLiftCommon
			\draw[->] (u1) -- (w);
			\draw[->] (u2) -- (w);
			\draw[->] (u3) -- (w);
		\end{tikzpicture}
		\caption{After: $\siblingsLift{\treeRootedAt{}}{\vertexV}$}
	\end{subfigure}
	\caption{The sibling lift at $\vertexV$ (in blue). Its children are
	moved from $\vertexV$ to its parent $\vertexW$ (in orange), so that
	they become siblings of $\vertexV$ under $\vertexW$.}
	\label{fig:sibling-lift}
\end{figure}
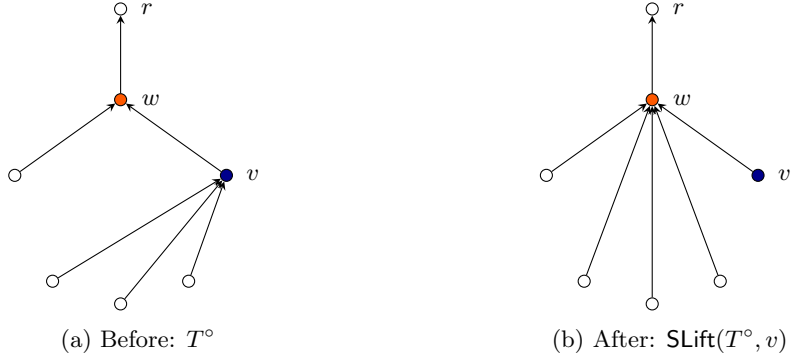

\begin{lemma}[Properties of the siblings lift operator]
	The sibling lift operator has the following properties:
	\begin{enumerate}
		\item $\siblingsLift{}{}$ preserves both the number of vertices and edges;
		\item $\siblingsLift{}{}$ preserves functionality; and
		\item $\siblingsLift{}{}$ is inflationary with respect to the coefficient-wise order:
		\begin{equation}\label{equ:lift-monotonicity}
			\generatingFunction{\siblingsLift{{\treeRootedAt{}}}{\vertexV}}{} \polyGreater \generatingFunction{{\treeRootedAt{}}}{}\enspace.
		\end{equation}
	\end{enumerate}
\end{lemma}
\begin{proof}
	The first two properties follow directly from the definition of the operator; thus, we only need to demonstrate the inflationary property of the sibling lift.
	Fix $\vertexV$ arbitrarily, $\vertexW \defeq {\treeRootedAt{}}(\vertexV)$ be its unique successor and $\vertexU_1, \dotsc, \vertexU_q$ be its predecessors.
	We also define
	\begin{equation*}
		\widehat{\treeRootedAt{}} \defeq \siblingsLift{{\treeRootedAt{}}}{\vertexV}\enspace,
	\end{equation*}
	and for every subtree $\treeRootedAt{\vertexU}$ of $\treeRootedAt{}$ rooted at $\vertexU$, we let $\widehat{\treeRootedAt{\vertexU}}$ be the subtree of $\widehat{\treeRootedAt{}}$ rooted at $\vertexU$.
	Furthermore, let $\alternativeCountingPolynomialU{\variable} \defeq \prod_{i = 1}^{j} \generatingFunction{{\treeRootedAt{\vertexU_i}}}{\variable} \polyGreater 1$ and $\alternativeCountingPolynomialW{\variable} \defeq \prod_{\vertexT \in \rev{\graphG}(\vertexW) \setminus \set*{\vertexV}} \generatingFunction{{\treeRootedAt{\vertexT}}}{\variable} \polyGreater 1$ be polynomials which are not affected by the lift.

	It follows from the recursive formulation of \zcref[S]{thm:successor-closed-counting-polynomial-for-functional-digraph} that since the change to $\treeRootedAt{}$ only occurs in the subtree $\treeRootedAt{\vertexW}$, if the new subtree has a coefficient-wise larger successor-closed counting polynomial, i.e., if their difference $\differenceAtW{}$ is coefficient-wise positive, then $\generatingFunction{\widehat{\treeRootedAt{}}}{} - \generatingFunction{{\treeRootedAt{}}}{}$ is coefficient-wise positive.
	From \zcref[S]{thm:successor-closed-counting-polynomial-for-functional-digraph}, we get the following two polynomials:
	\begin{align*}
		\generatingFunction{\widehat{\treeRootedAt{\vertexW}}}{\variable} &= 1 + \variable \generatingFunction{\widehat{\treeRootedAt{\vertexV}}}{\variable} \alternativeCountingPolynomialU{\variable} \alternativeCountingPolynomialW{\variable} \\
		&= 1 + \variable (1 + \variable) \alternativeCountingPolynomialU{\variable} \alternativeCountingPolynomialW{\variable}\enspace, \\
		\generatingFunction{{\treeRootedAt{\vertexW}}}{\variable} &= 1 + \variable \generatingFunction{{\treeRootedAt{\vertexV}}}{\variable} \alternativeCountingPolynomialW{\variable} \\
		&=1 + \variable \left(1 + \variable \alternativeCountingPolynomialU{\variable} \right) \alternativeCountingPolynomialW{\variable}\enspace,
	\end{align*}
	and their difference is 
	\begin{equation*}
		\differenceAtW{\variable} \defeq \variable \alternativeCountingPolynomialW{\variable} \left(\alternativeCountingPolynomialU{\variable} - 1\right)\polyGreater 0\enspace,
	\end{equation*}
	which implies \zcref[S]{equ:lift-monotonicity} and concludes the proof of this lemma.
\end{proof}

We now introduce the \emph{move lift operator}.
\begin{definition}
    Let $\cycleRootedtree{}$ be a cycle-rooted tree, and let $\vertexV$ be a vertex of the cycle with parent $\vertexW$.
    The \defemph{move lift at $\vertexV$}, denoted by $\moveLift{{\cycleRootedtree{}}}{\vertexV}$, is obtained by moving the children of $\vertexV$ not in the cycle from $\vertexV$ to its parent $\vertexW$: for every child $\vertexU$ of $\vertexV$ not in the cycle, the edge $\connect{\vertexU}{\vertexV}$ is replaced by $\connect{\vertexU}{\vertexW}$.
\end{definition}
\begin{remark}
    When the cycle-rooted tree is a star $\cycleRootedStar{\treeInternalSizes{}}$, the resulting graph is $\cycleRootedStar{\treeInternalSizes{}'}$, where $\treeInternalSizes{}'$ is obtained from $\treeInternalSizes{}$ by replacing the sizes associated with $\vertexV$ and $\vertexW$ by $\treeInternalSizes{}'_{\vertexV}=0$ and $\treeInternalSizes{}'_{\vertexW}=\treeInternalSizes{}_{\vertexV}+\treeInternalSizes{}_{\vertexW}$, respectively, and leaving all other sizes unchanged.
	We depict this case in \zcref[S]{fig:move-lift}.
\end{remark}

\begin{figure}[t]
	\centering
	\begin{subfigure}[b]{0.42\textwidth}
		\centering
		\begin{tikzpicture}[>=stealth]
			\moveLiftBefore
		\end{tikzpicture}
		\caption{Before: $\cycleRootedStar{\treeInternalSizes{}}$}
	\end{subfigure}
	\hfill
	\begin{subfigure}[b]{0.42\textwidth}
		\centering
		\begin{tikzpicture}[>=stealth]
			\moveLiftAfter
		\end{tikzpicture}
		\caption{After: $\moveLift{\cycleRootedStar{\treeInternalSizes{}}}{\vertexV}$}
	\end{subfigure}
	\caption{The move lift at the cycle vertex $\vertexV$ (in blue). The
	children of $\vertexV$ not on the cycle are moved to its cycle-parent
	$\vertexW$ (in orange); the cycle itself is unchanged, and $\vertexV$
	is left with no non-cycle children.}
	\label{fig:move-lift}
\end{figure}
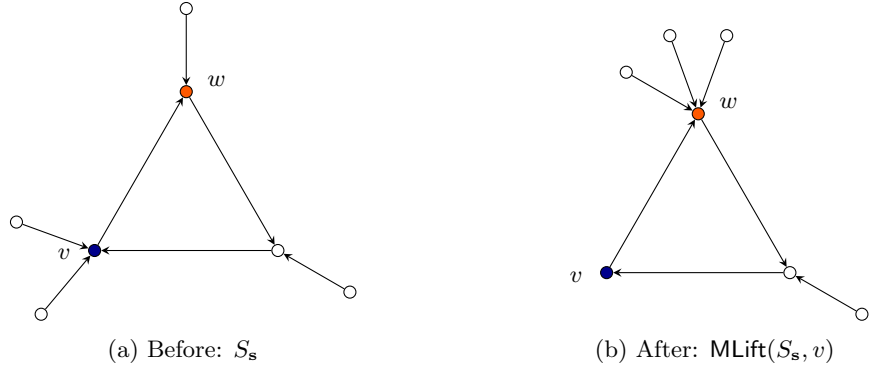

\begin{lemma}[Properties of the move lift operator]
	The move lift operator has the following properties:
	\begin{enumerate}
		\item $\moveLift{}{}$ preserves both the number of vertices and edges;
		\item $\moveLift{}{}$ preserves functionality; and
		\item $\moveLift{}{}$ is \defemph{$\generatingFunction{}{}$-invariant}, meaning that $\generatingFunction{\moveLift{{\cycleRootedtree{}}}{\vertexV}}{} = \generatingFunction{{\cycleRootedtree{}}}{}$. 
	\end{enumerate}
\end{lemma}

\begin{proof}
	The first two properties follow directly from the definition of the operator.
	For the third property, we want to prove that
	\begin{equation}\label{equ:lift-equality}
		\generatingFunction{\moveLift{{\cycleRootedtree{}}}{\vertexV}}{} = \generatingFunction{{\cycleRootedtree{}}}{}\enspace,
	\end{equation}
	for every vertex $\vertexV$ in the cycle.

	Let $\vertexV$ be any vertex in the cycle, $\vertexW \defeq {\cycleRootedtree{}}(\vertexV)$ be its unique successor and $\vertexU_1, \dotsc, \vertexU_q$ be its predecessors in $\vertexRootedtree{\vertexV}$, that is, not in the cycle.
	As before, we also define the polynomials $\alternativeCountingPolynomialU{\variable} \defeq \prod_{i = 1}^{j} \generatingFunction{{\treeRootedAt{\vertexU_i}}}{\variable} \polyGreater 1$ and $\alternativeCountingPolynomialW{\variable} \defeq \prod_{\vertexT \in \rev{\graphG}(\vertexW) \setminus \set*{\vertexV}} \generatingFunction{{\treeRootedAt{\vertexT}}}{\variable} \polyGreater 1$ which are also not affected by this lift.
	Finally, we let $\widehat{\cycleRootedtree{}} \defeq \moveLift{{\cycleRootedtree{}}}{\vertexV}$ be the transformed graph and define the difference polynomial $\differenceCycleRootedTree{} \defeq \generatingFunction{\widehat{\cycleRootedtree{}}}{} - \generatingFunction{{\cycleRootedtree{}}}{}$.

	Note that $\moveLift{}{}$ only affects the subtrees rooted at $\vertexV$ and $\vertexW$, and therefore, \zcref[S]{thm:successor-closed-counting-polynomial-for-functional-digraph} implies that
	\begin{equation}\label{equ:difference-proportion}
		\differenceCycleRootedTree{} \positiveProportional
		\pars*{\generatingFunction{\widehat{\treeRootedAt{\vertexV}}}{\variable} - 1} \cdot \pars*{\generatingFunction{\widehat{\treeRootedAt{\vertexW}}}{\variable} - 1}
		-
		\pars*{\generatingFunction{{\treeRootedAt{\vertexV}}}{\variable} - 1} \cdot \pars*{\generatingFunction{{\treeRootedAt{\vertexW}}}{\variable} - 1}
	\end{equation}
	Before the lifting, we have
	\begin{align*}
		\generatingFunction{{\treeRootedAt{\vertexV}}}{\variable} &= 1 + \variable \alternativeCountingPolynomialU{\variable}\enspace,\\
		\generatingFunction{{\treeRootedAt{\vertexW}}}{\variable} &= 1 + \variable \alternativeCountingPolynomialW{\variable}\enspace,
	\end{align*}
	and after,
	\begin{align*}
		\generatingFunction{\widehat{\treeRootedAt{\vertexV}}}{\variable} &=1 + \variable \enspace,\\
		\generatingFunction{\widehat{\treeRootedAt{\vertexW}}}{\variable} &= 1 + \variable \alternativeCountingPolynomialW{\variable} \alternativeCountingPolynomialU{\variable}\enspace.
	\end{align*}
	Injecting these quantities into \zcref[S]{equ:difference-proportion}, we obtain $\differenceCycleRootedTree{} \positiveProportional 0$, which implies $\differenceCycleRootedTree{} = 0$.
\end{proof}

We now have all the necessary tools to prove the two main theorems of this section.

\begin{proof}[{Proof of {\zcref[S]{thm:upper-bound-extremal}}}]
	We start by proving the upper bound, then the lower bounds, and we conclude with the generating functions. 

	\paragraph{Upper bound}
	First, note that any vertex-rooted tree ${\vertexRootedtree{}}$ can be transformed into ${\cycleRootedStar{\card*{\vertexRootedtree{}} - 1}}{}$ through a finite number of sibling lifts.
	The inflationary property of $\siblingsLift{}{}$ ensures $\generatingFunction{{\vertexRootedtree{}}}{} \polySmaller \generatingFunction{{\cycleRootedStar{\card*{\vertexRootedtree{}} - 1}}{}}{}$.
	Similarly, any cycle-rooted tree ${\cycleRootedtree{}}$, with subtrees of cardinalities $\vertexRootedTreeSize{1}, \dotsc, \vertexRootedTreeSize{\cycleLength{}}$, can be transformed into a cycle-rooted star ${\cycleRootedStar{\vertexRootedTreeSizes - \mathbf{1}}}$ with a larger generating function by applying the sibling lift a finite number of times to each of its vertex-rooted trees.
	Then, we obtain ${\concentratedCycleRootedStar{\cycleLength{}}{\cycleRootedTreeSize{} - \cycleLength{}}}$, where $\cycleRootedTreeSize{} \defeq \sum_{i = 1}^{\cycleLength{}} \vertexRootedTreeSize{i}$, without affecting the generating function, with $\cycleLength{} - 1$ move lifts.
	The characterization of functional digraph (\zcref[S]{pro:functional-graph-as-cycle-rooted-forest}) combined with \zcref[S]{lma:product-rule} implies the inequality $\graphG \polySmaller \graphH$.
	
	\paragraph{Lower bound}
	First, we transform every vertex-rooted tree $\vertexRootedtree{}$ into a path $\pathGraph{\vertexRootedTreeSize{} - 1}$ with a finite number of inverse sibling lifts.
	As before, the inflationary property ensures that $\generatingFunction{\pathGraph{\vertexRootedTreeSize{} - 1}}{} \polySmaller \generatingFunction{\vertexRootedtree{}}{}$.
	For a cycle-rooted tree $\cycleRootedtree{}$, we repeatedly apply the cycle extension to obtain a cycle $\cycle{\cycleRootedTreeSize{}}$ which minimizes the number of successor-closed subsets by \zcref[S]{lma:leaf-to-cycle}.
	Unlike the upper bound, the lower bound reduction (\zcref[S]{lma:leaf-to-cycle}) proceeds directly to the fully concentrated cycle, without requiring an intermediate stopping point analogous to $\cycleLength{} = 2$ above.
	
	\paragraph{Successor-closed generating functions}
	For the generating function of $\graphH$, we obtain it from \zcref[S]{cly:generating-function-vertex-rooted-star,cly:generating-function-concentrated-cycle-rooted-star} and \zcref[S]{lma:product-rule}.
	For $\graphF$, we use \zcref[S]{cly:generating-function-cycle,cly:generating-function-path} instead.
\end{proof}

\section{The extremal digraph for a fixed number of edges}
\label{sec:extremal-digraph}

In this section, we derive a lower and an upper bound for any functional digraph $\graphG$ where we only fix its number of vertices $\numberOfVertices$ and its number of edges $\numberOfEdges$.
We derive extremal graphs from the previous section and then optimize over all possible graphs conditioned on the fixed number of edges and vertices.

\begin{theorem}\label{thm:upper-bound-extremal-2}
	Let $\graphG$ be any functional digraph with $\numberOfVertices$ vertices and $\numberOfEdges < \numberOfVertices$ edges.
	Let
	\begin{equation*}
		\graphH \defeq \starGraph{\numberOfEdges} \disjointUnion \pars*{\largeDisjointUnion_{i = 1}^{\numberOfVertices - \numberOfEdges - 1} \starGraph{0}}\enspace.
	\end{equation*}
	Then, the inequality $\generatingFunction{\graphG}{} \polySmaller \generatingFunction{\graphH}{}$ holds, and $\graphH$ has successor-closed counting polynomial
	\begin{equation*}
		\generatingFunction{\graphH}{} = \pars*{1 + \variable}^{\numberOfVertices - \numberOfEdges - 1} \pars*{1 + \variable \pars*{1 + \variable}^{\numberOfEdges}}\enspace.
	\end{equation*}
	Moreover, when $\numberOfEdges > 1$, letting
	\begin{equation*}
		\graphF \defeq \cycle{\numberOfEdges} \disjointUnion \pars*{\largeDisjointUnion_{i = 1}^{\numberOfVertices - \numberOfEdges} \starGraph{0}}\enspace,
	\end{equation*}
	the inequality $\generatingFunction{\graphF}{} \polySmaller \generatingFunction{\graphG}{}$ holds, and $\graphF$ has successor-closed counting polynomial
	\begin{equation*}
		\generatingFunction{\graphF}{} = \pars*{1 + \variable}^{\numberOfVertices - \numberOfEdges}\pars*{1 + \variable^{\numberOfEdges}}\enspace.
	\end{equation*}
\end{theorem}

\begin{remark}
	When $\numberOfEdges = 1$, the graph $\graphG$ is unique up to isomorphism: exactly one vertex is active, so $\graphG \cong \starGraph{1} \disjointUnion \pars*{\largeDisjointUnion_{i=1}^{\numberOfVertices-2}\starGraph{0}}$, and its generating function is simply $\pars*{1 + \variable + \variable^2} \pars*{1 + \variable}^{\numberOfVertices - 2}$.
	Similarly, when $\numberOfEdges = 0$, $\graphG$ is forced to be $\numberOfVertices$ isolated vertices, with generating function $\pars*{1 + \variable}^\numberOfVertices$.
\end{remark}

The following two exchange lemmas break a cycle while preserving the number of edges and vertices.
\begin{lemma}\label{lma:cycle-to-vertex-rooted-tree}
	For every integer $a$ and $b$,
	\begin{equation*}
		\generatingFunction{\starGraph{a + b + 2}}{}
		\polyGreater
		\generatingFunction{\concentratedCycleRootedStar{2}{a}}{} \cdot \generatingFunction{\starGraph{b}}{}\enspace.
	\end{equation*}
	Note that the number of vertices and edges is preserved.
\end{lemma}
\begin{proof}
	From \zcref[S]{cly:generating-function-vertex-rooted-star,cly:generating-function-concentrated-cycle-rooted-star}, we obtain that their difference is
	\begin{align*}
		\Delta(a, b) &\defeq 1 + \variable\pars*{1 + \variable}^{a + b + 2}
		-
		\pars*{1 + \variable^2 \pars*{1 + \variable}^{a}} \pars*{1 + \variable \pars*{1 + \variable}^{b}} \\
		&= \variable\pars*{1+\variable}^{a+b+2} - \variable\pars*{1+\variable}^b - \variable^2\pars*{1+\variable}^a - \variable^3\pars*{1+\variable}^{a+b}\enspace.
	\end{align*}
	Fix the integer $b$ arbitrarily. We show, by induction on $a$, that $\Delta(a,b) \polyGreater 0$.

	\paragraph{Base case ($a = 0$)}
	\begin{align*}
		\Delta(0,b) &= \variable\pars*{1+\variable}^{b+2} - \variable\pars*{1+\variable}^b - \variable^2 - \variable^3\pars*{1+\variable}^b \\
		&= \variable\pars*{1+\variable}^b\left[\pars*{1+\variable}^2 - 1 - \variable^2\right] - \variable^2 \\
		&= 2\variable^2\pars*{1+\variable}^b - \variable^2 \\
		&= \variable^2\left[2\pars*{1+\variable}^b - 1\right] \\
		&\polyGreater 0\enspace,
	\end{align*}
	since $2\pars*{1+\variable}^b \polyGreater 1$.

	\paragraph{Inductive step ($a > 0$)}
	First, note that
	\begin{align*}
		\Delta(a,b) - \Delta(a-1,b)
		&= \variable^2\pars*{1+\variable}^{a-1}\left[\pars*{1+\variable}^{b+2} - \variable - \variable^2\pars*{1+\variable}^b\right] \\
		&= \variable^2\pars*{1+\variable}^{a-1}\left[\variable\left(\pars*{1+\variable}^b\pars*{1+\variable+\variable^2} - 1\right) + \variable\right] \\
		&\polyGreater 0\enspace,
	\end{align*}
	since $\pars*{1+\variable}^b\pars*{1+\variable+\variable^2} \polyGreater 1$.
	Then,
	\begin{equation*}
		\Delta(a,b) = \Delta(a,b) - \Delta(a-1,b) + \Delta(a-1,b) \polyGreater \Delta(a-1,b) \polyGreater 0\enspace,
	\end{equation*}
	by the induction hypothesis. This concludes the proof by induction and, therefore, the proof of the lemma.
\end{proof}

\begin{lemma}\label{lma:cycle-to-path}
	For every integer $\pathLength > 1$,
	\begin{equation*}
		\generatingFunction{\pathGraph{\pathLength}}{}
		\polyGreater
		\generatingFunction{\cycle{\pathLength}}{} \cdot \generatingFunction{\starGraph{0}}{}\enspace.
	\end{equation*}
	Note that the number of vertices and edges is preserved.
\end{lemma}
\begin{proof}
	By direct computation and from \zcref[S]{cly:generating-function-cycle,cly:generating-function-path,cly:generating-function-vertex-rooted-star},
	\begin{align*}
		\generatingFunction{\pathGraph{\pathLength}}{}
		-
		\generatingFunction{\cycle{\pathLength}}{} \cdot \generatingFunction{\starGraph{0}}{}
		&=\sum_{i=0}^{\pathLength + 1} \variable^i
		-
		\pars*{1 + \variable^\pathLength} \pars*{1 + \variable} \\
		&=\sum_{i=0}^{\pathLength + 1} \variable^i
		-
		\pars*{1 + \variable + \variable^\pathLength + \variable^{\pathLength + 1}} \\
		&\polyGreater 0\enspace,
	\end{align*}
	since $\pathLength > 1$.
\end{proof}

The following result compares two cycles with the larger merged cycle.
\begin{lemma}\label{lma:two-cycle-to-one-cycle}
	For every integer $a, b > 1$,
	\begin{equation*}
		\generatingFunction{\cycle{a}}{} \cdot \generatingFunction{\cycle{b}}{}
		\polyGreater
		\generatingFunction{\cycle{a + b}}{}\enspace.
	\end{equation*}
	Note that the number of vertices and edges is preserved.
\end{lemma}
\begin{proof}
	By direct computation and from \zcref[S]{cly:generating-function-cycle},
	\begin{align*}
		\generatingFunction{\cycle{a}}{} \cdot \generatingFunction{\cycle{b}}{}
		-
		\generatingFunction{\cycle{a + b}}{}
		&=\pars*{1 + \variable^a}\pars*{1 + \variable^b}
		-
		\pars*{1 + \variable^{a + b}} \\
		&=\variable^a + \variable^b \\
		&\polyGreater 0\enspace.
	\end{align*}
\end{proof}
\noindent We have a similar result for paths.
\begin{lemma}\label{lma:merging-path}
	For every integer $a, b > 1$,
	\begin{equation*}
		\generatingFunction{\pathGraph{a}}{} \cdot \generatingFunction{\pathGraph{b}}{}
		\polyGreater
		\generatingFunction{\pathGraph{a + b}}{}\enspace.
	\end{equation*}
	Note that the number of vertices and edges is preserved.
\end{lemma}
\begin{proof}
	By \zcref[S]{cly:generating-function-path}, we have
	\begin{equation*}
		\generatingFunction{\pathGraph{a}}{\variable} \cdot \generatingFunction{\pathGraph{b}}{\variable}
		-
		\generatingFunction{\pathGraph{a+b}}{\variable} = \pars*{\sum_{i=0}^{a + 1} \variable^i} \cdot \pars*{\sum_{i=0}^{b + 1} \variable^i} - \sum_{i=0}^{a + b + 1} \variable^i
		\polyGreater 0\enspace,
	\end{equation*}
	since for every integer $j \leq a + b + 1$, the coefficient of $\variable^j$ of the left term is at least $1$.
\end{proof}

We can finally prove the main theorem of this section.

\begin{proof}[Proof of {\zcref[S]{thm:upper-bound-extremal-2}}]
	Let $\graphG$ be any functional digraph with $\numberOfVertices$ vertices and $\numberOfEdges < \numberOfVertices$ edges.
	We prove the upper bound and its closed form, then the lower bound and its closed form.
	
	We obtain, from \zcref[S]{thm:upper-bound-extremal}, a graph $\graphH$ satisfying $\generatingFunction{\graphH}{} \polyGreater \generatingFunction{\graphG}{}$.
	Since every cycle length of $\graphH$ is at least $2$, letting $\cycleLengthVector$ be the constant vector $\pars*{2, \dotsc, 2}$ of length $\numberOfCycleRootedTree$, \zcref[S]{prop:cycle-length-space-monotone} gives $\graphI \defeq \graphIOf{\cycleLengthVector}$, which satisfies $\generatingFunction{\graphIOf{\cycleLengthVector}}{} \polyGreater \generatingFunction{\graphH}{} \polyGreater \generatingFunction{\graphG}{}$.
	This choice of $\cycleLengthVector$ ensures that the number of edges is preserved.

	Now we use \zcref[S]{lma:cycle-to-vertex-rooted-tree} to transform every pair of concentrated cycle-rooted star and vertex-rooted star in $\graphI$ into a single vertex-rooted star.
	This works since our assumption $\numberOfEdges < \numberOfVertices$ ensures there is at least one vertex-rooted star, counting $\starGraph{0}$.
	Indeed, $\numberOfEdges = \numberOfVertices$ occurs if and only if all connected components are cycle-rooted trees.
	Let $\graphJ$ be the resulting graph.
	\zcref[S]{lma:product-rule} ensures that $\graphJ \polyGreater \graphI$.
	Furthermore, it has the same number of vertices and edges as $\graphI$, and thus, as $\graphG$.
	Note that $\graphJ$ is now completely determined by a tuple of positive integers $\treeSizes{}$.
	We make this representation explicit using the notation:
	\begin{equation*}
		\parameterizedGraph{\treeSizes{}} \defeq \largeDisjointUnion_{i = 1}^{\card*{\treeSizes{}}} \starGraph{\treeSize{i} - 1}\enspace.
	\end{equation*}
	From \zcref[S]{cly:generating-function-vertex-rooted-star}, its generating function is
	\begin{equation*}
		\objective{\treeSizes{}} \defeq \generatingFunction{\parameterizedGraph{\treeSizes{}}}{\variable} = \prod_{i = 1}^{\card*{\treeSizes{}}} \pars*{1 + \variable \pars*{1 + \variable}^{\treeSize{i} - 1}}\enspace.
	\end{equation*}
	We also define the function $\edgeCount{}$, which counts the number of edges in $\parameterizedGraph{\treeSizes{}}$.
	Formally, we define
	\begin{equation*}
		\edgeCount{} \colon \treeSizes{} \mapsto \sum_{i=1}^{\card*{\treeSizes{}}} \pars*{\treeSize{i} - 1}\enspace.
	\end{equation*}

	The final step is to optimize with respect to the fixed number of edges and vertices.
	In symbols, we have to find a tuple $\maximalTreeInnerSizes$ solving the following optimization problem
	\begin{equation*}
		\max_{\substack{\treeSizes{} \partitionsOf \numberOfVertices \\ \edgeCount{\treeSizes{}}  = \numberOfEdges}} \objective{\treeSizes{}}\enspace,
	\end{equation*}
	over the coefficient-wise order.
	The integer partition constraint implies $\edgeCount{\treeSizes{}} = \numberOfVertices - \card*{\treeSizes{}}$ and therefore, the edge constraint is equivalent to
	\begin{equation}\label{equ:partition-size}
		\card*{\treeSizes{}} = \numberOfVertices - \numberOfEdges\enspace. 
	\end{equation}

	Pick two distinct indices $i,j \in [\card*{\treeSizes{}}]$ such that $1 < \treeSize{i} \leq \treeSize{j}$.
	Note that such elements always exist since \zcref[S]{equ:partition-size} holds.
	We prove that
	\begin{equation}\label{equ:sparce-is-better}
		\objective{\treeSizes{}}\big|_{\treeSize{i} \mapsto \treeSize{i} - 1,\; \treeSize{j} \mapsto \treeSize{j} + 1} \;\polyGreater\; \objective{\treeSizes{}}\enspace.
	\end{equation}
	Informally, a smaller element of the partition can give a unit to a larger one while only increasing the objective.
	Assuming \zcref[S]{equ:sparce-is-better} holds, it is easy to see that 
	\begin{equation*}
		\maximalTreeInnerSizes = (\numberOfEdges + 1, \underbrace{1, \dotsc, 1}_{\mathclap{\text{$\numberOfVertices - \numberOfEdges - 1$ times}}})\enspace,
	\end{equation*}
	as we can always reach a permutation of this tuple after finitely many applications of this transformation. 
	Hence, proving \zcref[S]{equ:sparce-is-better} is enough to conclude the upper bound part.

	Let $\smallerElement \defeq \treeSize{i} - 1$ and $\largerElement \defeq \treeSize{j} - 1$ be two integers.
	From \zcref[S]{cly:generating-function-vertex-rooted-star},
	\begin{align*}
		\generatingFunction{\starGraph{\largerElement+1}}{\variable} \cdot \generatingFunction{\starGraph{\smallerElement-1}}{\variable}
		-\generatingFunction{\starGraph{\largerElement}}{\variable} \generatingFunction{\starGraph{\smallerElement}}{\variable}
		&=
		\left(1+\variable(1+\variable)^{\largerElement+1}\right)
		\left(1+\variable(1+\variable)^{\smallerElement-1}\right)
		\\
		&\quad-
		\left(1+\variable(1+\variable)^\largerElement\right)
		\left(1+\variable(1+\variable)^\smallerElement\right)
		\\
		&\positiveProportional
		(1+\variable)^{\largerElement}
		\left(1 + \variable - 1\right) \\
		&\quad -(1+\variable)^{\smallerElement - 1}
		\left(1 + \variable - 1\right) \\
		&=\variable (1 + \variable)^{\largerElement} - \variable (1 + \variable)^{\smallerElement - 1} \\
		&\positiveProportional (1 + \variable)^{\largerElement - (\smallerElement - 1)} - 1 \\
		&= (1 + \variable)^{\largerElement - \smallerElement + 1} - 1 \\
		&\polyGreater 0\enspace,
	\end{align*}
	where the last inequality follows from $\largerElement \geq \smallerElement$.

	For the lower bound, we first use \zcref[S]{thm:upper-bound-extremal} to get $\graphF$ such that $\generatingFunction{\graphF}{} \polySmaller \generatingFunction{\graphG}{}$.
	Then, we use \zcref[S]{lma:merging-path} to merge all paths into a single path.
	The assumption $\numberOfEdges > 1$ ensures that the resulting path has length $\pathLength > 1$ or that there is a cycle with length $\cycleLength{} > 1$.
	In the first case, we transform the single path into a cycle of length $\pathLength$ and an isolated vertex while preserving the number of edges and vertices using \zcref[S]{lma:cycle-to-path}.
	Therefore, both cases result in a graph with only cycles and bare vertices.
	Applying \zcref[S]{lma:two-cycle-to-one-cycle} repeatedly, we merge all resulting cycles into one.
	The desired graph follows from the constraints on the number of vertices and edge and its generating function from \zcref[S]{cly:generating-function-cycle} and \zcref[S]{lma:product-rule}.
\end{proof}

\section{A probabilistic variant of the Bollobás set-pair inequality}
\label{sec:probabilistic-bollobas}

In this section, we prove two probabilistic extensions of set-pair results.

\begin{theorem}
	\label{thm:probabilistic-combinatorial-lemma}
	Let $\setA_1, \dotsc, \setA_\numberOfPairs, \setB_1, \dotsc, \setB_\numberOfPairs$ be some subsets of some universe satisfying $\setA_i \cap \setB_i = \varnothing$ and $\card{\setA_i} + \card{\setB_i} \leq \setsSize$. Then, if $\numberOfPairs > \binom{\setsSize}{\setsSize/2}$, we have
	\begin{align*}
		\Pr_{\indicesSet \leftarrow \binom{[\numberOfPairs]}{\indicesSetSize}}\left[ \exists i \in \indicesSet \quad  \exists j \in [\numberOfPairs] \setminus \indicesSet \quad \setA_i \cap \setB_j \subseteq \setB_i\right] &\geq 1 - \prod_{i = 0}^{\numberOfEdges} \frac{\numberOfPairs - \indicesSetSize - i}{\numberOfPairs - i} - \frac{\indicesSetSize}{\numberOfPairs} \\
		&\geq 1 - \exp\left(-\frac{\indicesSetSize(\numberOfEdges+1)}{\numberOfPairs}\right) - \frac{\indicesSetSize}{\numberOfPairs}\enspace,
	\end{align*}
	where $\numberOfEdges \defeq \numberOfPairs - \binom{\setsSize}{\setsSize/2}$.
\end{theorem}

Before proving this theorem, we present a variant based on the \emph{skew inequality}.
\begin{corollary}[A probabilistic skew Bollobás inequality]
	\label{cor:skew-probabilistic-bollobas}
	Let $\setsSize \in \N$ and let $\setA_1, \dotsc, \setA_\numberOfPairs, \setB_1, \dotsc, \setB_\numberOfPairs \subseteq [\setsSize]$ satisfy $\setA_i \cap \setB_i = \varnothing$ for every $i$.
	Then, if $\numberOfPairs > \pars*{\setsSize+1}\binom{\setsSize}{\setsSize/2}$, we have
	\begin{align*}
		\Pr_{\indicesSet \leftarrow \binom{[\numberOfPairs]}{\indicesSetSize}}
		\left[ \exists i \in \indicesSet \ \exists j \in [\numberOfPairs] \setminus \indicesSet,\ i < j
		\quad \setA_i \cap \setB_j \subseteq \setB_i\right]
		&\geq 1 - \prod_{i = 0}^{\numberOfEdges} \frac{\numberOfPairs - \indicesSetSize - i}{\numberOfPairs - i} - \frac{\indicesSetSize}{\numberOfPairs} \\
		\geq 1 - \exp\left(-\frac{\indicesSetSize(\numberOfEdges+1)}{\numberOfPairs}\right) - \frac{\indicesSetSize}{\numberOfPairs}\enspace,
	\end{align*}
	where $\numberOfEdges \defeq \numberOfPairs - \pars*{\setsSize+1}\binom{\setsSize}{\setsSize/2}$.
\end{corollary}

We now prove the main theorem.

\begin{proof}[Proof of {\zcref[S]{thm:probabilistic-combinatorial-lemma}}]
	Suppose $\setA_1, \dotsc, \setA_\numberOfPairs, \setB_1, \dotsc, \setB_\numberOfPairs$ satisfy the assumption of \zcref[S]{thm:probabilistic-combinatorial-lemma}.
	Consider the digraph $\graphG$ over vertices $[\numberOfPairs]$ such that there is an edge between $i$ and $j$ if and only if $ \setA_i \cap \setB_j \subseteq \setB_i$.
	We know from \zcref[S]{lma:combinatorial-lemma} that there are at least $\numberOfEdges \defeq \numberOfPairs - \binom{\setsSize}{\setsSize/2}$ indices $i \in [\numberOfPairs]$ such that there exists a distinct $j \in [\numberOfPairs]$ such that $ \setA_i \cap \setB_j \subseteq \setB_i$ (otherwise we would contradict \zcref[S]{lma:combinatorial-lemma}).
	In other words, $\graphG$ has at least $\numberOfEdges$ active vertices.

	Using \zcref[S]{lma:monotonicity-under-edge-deletion}, we remove edges until there are only $\numberOfEdges$ active vertices with outgoing degree one.
	We call the resulting digraph $\graphH$.
	Observe that $\graphH$ is functional, has $\numberOfEdges$ edges, $\numberOfPairs$ vertices and satisfies $\generatingFunction{\graphH}{} \polyGreater \generatingFunction{\graphG}{}$.

	For every $\indicesSetSize \in [\numberOfPairs]$, the upper bound of \zcref[S]{thm:upper-bound-extremal-2} implies that the number of successor closed $\indicesSetSize$-subsets is at most
	\begin{align*}
		\coeff*{\variable^\indicesSetSize} \generatingFunction{\graphH}{\variable} &= \coeff*{\variable^\indicesSetSize} \pars*{1 + \variable}^{\numberOfVertices - \numberOfEdges - 1} \pars*{1 + \variable \pars*{1 + \variable}^{\numberOfEdges}} \\ 
		&= \coeff*{\variable^\indicesSetSize} \pars*{1 + \variable}^{\numberOfVertices - \numberOfEdges - 1} 
		+
		 \coeff*{\variable^{\indicesSetSize - 1}} \pars*{1 + \variable}^{\numberOfVertices - 1}\\
		&= \binom{\numberOfVertices - \numberOfEdges - 1}{\indicesSetSize} + \binom{\numberOfVertices - 1}{\indicesSetSize - 1}\enspace.
	\end{align*} 
	Thus, the probability that we get a set $\indicesSetSize$-size subset $\indicesSet$ with at least one outgoing edge, or equivalently,
	\begin{equation*}
		\exists i \in \indicesSet \quad  \exists j \in [\numberOfPairs] \setminus \indicesSet \quad \setA_i \cap \setB_j \subseteq \setB_i\enspace,
	\end{equation*}
	is at least
	\begin{align*}
		1 - \frac{\binom{\numberOfVertices - \numberOfEdges - 1}{\indicesSetSize} + \binom{\numberOfVertices - 1}{\indicesSetSize - 1}}{\binom{\numberOfVertices}{\indicesSetSize}}
		&= 1 - \frac{\binom{\numberOfVertices-\numberOfEdges-1}{\indicesSetSize}}{\binom{\numberOfVertices}{\indicesSetSize}} - \frac{\binom{\numberOfVertices-1}{\indicesSetSize-1}}{\binom{\numberOfVertices}{\indicesSetSize}} \\
		&= 1 - \prod_{i=0}^{\numberOfEdges}\frac{\numberOfVertices-\indicesSetSize-i}{\numberOfVertices-i} - \frac{\indicesSetSize}{\numberOfVertices}\enspace,
	\end{align*}
	as desired.

	The second inequality is obtained by noticing that
	\begin{align*}
		\prod_{i = 0}^{\numberOfEdges} \frac{\numberOfPairs - \indicesSetSize - i}{\numberOfPairs - i}
		&= \frac{\binom{\numberOfPairs - \indicesSetSize}{\numberOfEdges + 1}}{\binom{\numberOfPairs}{\numberOfEdges + 1}}\\
		&= \prod_{i = 0}^{\numberOfEdges} \pars*{1 - \frac{\indicesSetSize}{\numberOfPairs - i}}
		\leq \prod_{i = 0}^{\numberOfEdges} \exp\left(-\frac{\indicesSetSize}{\numberOfPairs - i}\right)
		= \exp\left(-\indicesSetSize\sum_{i=0}^{\numberOfEdges}\frac{1}{\numberOfPairs-i}\right)\\
		&\leq \exp\left(-\frac{\indicesSetSize(\numberOfEdges+1)}{\numberOfPairs}\right)\enspace,
	\end{align*}
	using, respectively, the telescoping falling-factorial identity, $1-t\le e^{-t}$ applied termwise, and $\numberOfPairs - i \leq \numberOfPairs$ for every $i \in \set*{0,\dotsc,\numberOfEdges}$.
\end{proof}

Finally, we prove the corollary.

\begin{proof}[Proof of {\zcref[S]{cor:skew-probabilistic-bollobas}}]
	Identical to the proof of \zcref[S]{thm:probabilistic-combinatorial-lemma}, with \zcref[S]{lma:skew-combinatorial-lemma} in place of \zcref[S]{lma:combinatorial-lemma}: build the digraph $G$ on $[\numberOfPairs]$ with an edge $i\to j$ whenever $i<j$ and $\setA_i\cap\setB_j\subseteq\setB_i$.
	By construction, every edge satisfies $i<j$, so $G$ is acyclic; \zcref[S]{lma:skew-combinatorial-lemma} guarantees at least $\numberOfEdges$ active vertices, and the remainder of the argument (edge deletion to functionality, upper bound of \zcref[S]{thm:upper-bound-extremal-2}, the binomial telescoping) proceeds exactly as before.
\end{proof}

\printbibliography

\end{document}